\documentclass[english]{smfart}
\usepackage{sabbah_grenoble10}

\begin{document}
\frontmatter
\title{Non-commutative Hodge structures}

\author[C.~Sabbah]{Claude Sabbah}
\address{UMR 7640 du CNRS\\
Centre de Mathématiques Laurent Schwartz\\
École polytechnique\\
F--91128 Palaiseau cedex\\
France}
\email{sabbah@math.polytechnique.fr}
\urladdr{http://www.math.polytechnique.fr/~sabbah}

\thanks{This research was supported by the grant ANR-08-BLAN-0317-01 of the Agence nationale de la recherche.}

\begin{abstract}
This article gives a survey of recent results on a generalization of the notion of a Hodge structure. The main example is related to the Fourier-Laplace transform of a variation of polarizable Hodge structure on the punctured affine line, like the Gauss-Manin systems of a proper or tame algebraic function on a smooth quasi-projective variety. Variations of non-commutative Hodge structures often occur on the tangent bundle of Frobenius manifolds, giving rise to a tt* geometry.
\end{abstract}



\alttitle{Structures de Hodge non commutatives}
\begin{altabstract}
Nous donnons un panorama des résultats récents concernant une généralisation de la notion de structure de Hodge. L'exemple principal est celui produit par la transformation de Fourier-Laplace d'une variation de structure de Hodge polarisable sur la droite affine épointée, comme les systèmes de Gauss-Manin de fonctions algébriques propres ou modérées sur une variété quasi-projective lisse complexe. Le fibré tangent d'une variété de Frobenius peut souvent être muni d'une variation de structures de Hodge non-commutatives polarisables, d'où l'on déduit une géométrie spéciale du type tt*.
\end{altabstract}

\maketitle
\mainmatter%

\section{Introduction}

\begin{notation}\label{notation}
All along this article, the abbreviation ``nc'' means ``non-commutative''. We will consider the Riemann sphere $\PP^1$ equipped with two affine charts $U_0=\Spec\CC[\hb]$, $U_\infty=\Spec\CC[\hb']$, and $U_0\cap U_\infty=\Spec\CC[\hb,\hbm]=\Spec\CC[\hb',\hb^{\prime-1}]$ where the identification is given by $\hb'=\nobreak\hbm$. We will consider the involution $\iota:\PP^1\to\PP^1$ defined by $\iota(\hb)=-\hb$, and the anti-holomorphic involution $\gamma:\PP^1\to\ov\PP^1$ defined by $\gamma(\hb)=1/\ov\hb$. The composed involution $\gamma\circ\iota$ will be denoted by $\sigma$. When restricted to $\bbS\defin\{\hb\in U_0\mid|\hb|=1\}$, $\iota$ and $\sigma$ coincide, since $\gamma_{|\bbS}=\id_{\bbS}$.
\end{notation}

\num{}The terminology ``non-commutative Hodge structure'' (which should not be confused with that of non-abelian Hodge theory developed by C\ptbl Simpson \cite{Simpson91,Simpson97b,Simpson02}) has been introduced by Katzarkov-Kontsevich-Pantev \cite{K-K-P08} to cover the kind of Hodge structure one should expect on the periodic cyclic cohomology of smooth compact non-commutative spaces (although we will keep the setting of standard commutative algebraic geometry in this article and we will not provide any example in the setting of non-commutative spaces envisioned in \loccit, which we refer to for details, as well as to \cite{Kaledin08} and the recent preprints \cite{Kaledin10,Shklyarov11}).

Generalizations of the notion of a Hodge structure already occur for usual algebraic geometric objects, like polynomial functions, where such a structure is related to oscillating integrals. In this context, such a structure has been brought to light by Cecotti, Vafa et al\ptbl in various articles \cite{C-V91,C-F-I-V92,C-V93}, and formalized by C\ptbl Hertling in \cite{Hertling01} under the name of pure (and polarized) TERP structure, so as to treat variations of such objects. When such a structure occurs on the tangent bundle of a manifold, this produces a tt* geometry on this manifold.

By mirror symmetry, such a structure is expected on the quantum cohomology of some algebraic varieties, in the sense that the associated Frobenius manifold should underlie a variation of polarized non-commutative Hodge structures on its tangent bundle. Hertling (\cf \loccit) has also formalized the compatibility relations between the Frobenius manifold relations and the tt* geometry under the name of a C(ecotti)D(ubrovin)V(afa)-structure. Explicit formulas from the quantum cohomology point of view have been obtained by H\ptbl Iritani \cite{Iritani09,Iritani09b} (\cf also \cite{K-K-P08} for the projective space). We will not go further in this direction.

The purpose of this article is to survey recent developments concerning non-commutative Hodge structures by themselves. We will only mention some results concerning their variations and the corresponding limit theorems. One can already emphasize that these structures allow limiting behaviour with irregular singularities (wild behaviour), while usual variations of Hodge structures only allow limiting behaviour with regular singularities (tame behaviour). As a consequence, they fit with transformations which do not preserve the tame behaviour, like the Fourier-Laplace transformation.

The first attempt to provide a good Hodge theory in such a context seems to be the notes \cite{Deligne8406} by P\ptbl Deligne (the relation with \ncHodge structures is briefly discussed in \S\ref{sec:tamefunct} and with more details in \cite[\S6]{Bibi08}), and the expectation of a behaviour by Fourier transformation analogous to the case of Fourier-Deligne transformation in $\ell$-adic theory is clearly stated in \cite[Rem\ptbl7.3.3.3]{K-L85}. In some sense, Theorem \ref{th:FLTQ} answers this remark, at least in dimension one. On the other hand, the frame of such a theory, restricted to the regular singularity case, was in germ in the work of A.N\ptbl Varchenko \cite{Varchenko82} in Singularity theory, as well as in various later works concerning Hodge theory for isolated singularity of holomorphic functions.

\num{}
The underlying analytic theorems in this non-commutative Hodge theory apply to a wider context, that of \textit{twistor structures}. The main guiding article in this direction is due to C\ptbl Simpson \cite{Simpson97}, which unifies under the notion of twistor structure (and variation of such) various objects considered in nonabelian Hodge theory. The theory of variations of pure and polarized twistor structures has been completed by T\ptbl Mochizuki \cite{Mochizuki07} in the tame case, extending the previous work of Simpson in dimension one \cite{Simpson90}, and this work culminates with \cite{Mochizuki08}, where T\ptbl Mochizuki extends his previous results to the wild setting (see also \cite{Bibi01c,Bibi06b}).

Giving a pure and polarized twistor structure of weight $w\in\ZZ$ is equivalent to giving a complex vector space $H$ and a positive definite Hermitian form $h$ on it, together with an integer~$w$ (although giving a variation of such objects is more subtle, since it involves the notion of a harmonic metric \cite{Simpson92}). This can be encoded in a vector bundle of pure slope~$w$ on $\PP^1$ (which is now called a pure twistor structure of weight~$w$), together with the right replacement of a Hermitian form, called a polarization of this twistor structure.

\num{}
The next step consists in generalizing the notion of a polarized \textit{complex} Hodge structure. The bigrading is replaced with a meromorphic connection on the vector bundle on $\PP^1$ considered above, having a pole of order at most two at $0$ and $\infty$, and no other pole. The reason for restricting to poles of order two is explained in \S\ref{subsec:filteredspaces}. A vector bundle on $\PP^1$ with such a connection is called an integrable twistor structure (integrable because, when considering variations, it consists to adding an integrability condition to variations of twistor structures). It can also be described by linear algebra objects on $H$ (\cf \S\ref{subsec:linalg}), that is, endomorphisms $\cU,\cQ$ of $H$. However, variations of such objects produce non-trivial integrability conditions on these endomorphisms (\cf \cite{Hertling01}, \cite[Chap\ptbl7]{Bibi01c}). The endomorphism $\cQ$ is self-adjoint with respect to the metric $h$ and its eigenvalues play the role of the Hodge exponent~$p$ (more precisely, $p-w/2$) in $H^{p,w-p}$. However, these eigenvalues may vary in variations of pure polarizable integrable twistor structure (that we now call polarized pure complex \ncHodge structures).

\num{}
The main object in this article is the notion of a polarized \nckHodge structure, when $\kk$ is a subfield of $\RR$, \eg $\kk=\QQ$. When $\kk=\RR$, it corresponds to the notion of a pure polarized TERP structure as defined in \cite{Hertling01}, if one moreover takes care of a $\RR$-structure on the Stokes data of the connection, in case the singularity at~$0$ and~$\infty$ is irregular. We will recall the well-known analysis that we need of connections with a pole of order at most two in~\S\ref{subsec:Brieskorn}. The $\kk$-structure is included at the level of generalized monodromy data, \ie monodromy and (in the wild case) Stokes data. This way of treating the $\kk$-structure has been much generalized by T\ptbl Mochizuki in \cite{Mochizuki10} for arbitrary holonomic $\cD$-modules.

While one can define the notion of a pure \ncHodge structure without introducing a polarization (\cf\S\ref{num:nckHodge}), proving that a given set of data form a pure \ncHodge structure often uses the supplementary existence of a polarization. A simple, but nontrivial, example is given in \S\ref{sec:DM}, confirming a conjecture of C\ptbl Hertling and Ch\ptbl Sevenheck in~\cite{H-S06}, that they proved in some special cases. The main argument (\cf \cite{H-S09}) relies on a general way to produce polarized complex Hodge structures, which is given in \cite{Bibi05}. Namely, given a variation of polarized complex Hodge structure on the complex punctured affine line (a finite number of points deleted), one can associate, by Fourier-Laplace transformation, a polarized complex \ncHodge structure. This is reviewed in \S\ref{sec:FLT}, where we also give some complements for pure polarized \nckHodge structures. The same idea is used (\cf \cite{Bibi05}) to prove that the Brieskorn lattice of a regular function on an affine manifold with isolated singularities and a tame behaviour of the fibres at infinity (\eg a convenient and non-degenerate (Laurent) polynomial) underlies a rational \ncHodge structure. Results of this kind are reviewed in \S\ref{sec:tamefunct}. Meanwhile, we compare in \S\ref{num:expHS} the notion of \ncHodge structure with that of an exponential pure Hodge structure, introduced in \cite{K-S10}.

\num{}
Lastly, we consider numerical invariants of \ncHodge structures. There are various ways to replace the exponent~$p$ in $H^{p,q}$ (Hodge structure of weight~$w$, so that $q=w-p$). These are called the spectral numbers at $\hb=0$, the spectral numbers at $\hb=\infty$, and the supersymmetric index. The latter may vary in a real analytic way in variations of \ncHodge structures, hence is difficult to compute in general, being of a very transcendental nature. On the other hand, the definition of spectral numbers at $\hb=0$ goes back to Varchenko \cite{Varchenko82} and Steenbrink (\cf \cite{Steenbrink87}) in Singularity theory, and the spectral numbers at $\hb=\infty$ were introduced in \cite{Bibi96b} (\cf also \cite{Bibi96bb}). Many authors have considered these invariants for local or global singularities. For TERP$(w)$ structures, they should remain constant when defining classifying spaces (\cf \cite{H-S07, H-S08b,H-S08}), an important question when considering associated period mappings.

\num{}
In order to keep reasonable length to this paper, we do not treat with details variations of \ncHodge structures and their limits. We would like to emphasize, however, that the many-variable nilpotent orbit theorem of \cite{C-K-S86} and \cite{Kashiwara85} has now a ``wild twistor'' counterpart \cite{Mochizuki08}, as well as a TERP counterpart \cite{H-S07}, \cite{Mochizuki08b}. On the other hand, the $\QQ$-structure can also been considered, according to \cite{Mochizuki10} (see also \cite{Bibi10}).

\nums{Acknowledgements}
I thank Claus Hertling for reading a preliminary version of the manuscript and useful comments. The content of this survey article owes much to discussions and collaboration with him. I also thank Takuro Mochizuki, Christian Sevenheck and Jean-Baptiste Teyssier for their questions and suggestions, and for the various discussions we had.

\section{Connections with a pole of order two}\label{subsec:Brieskorn}

\paragraph{Filtered vector spaces with automorphism}\label{subsec:filteredspaces}
We will regard connections with a pole of order (at most) two as a suitable generalization of the notion of a filtered complex vector space equipped with an automorphism in the following sense. Let $V$ be a finite dimensional complex vector space, and let $\rT$ be an automorphism of $V$. We denote by $\rT_\mathrm{s}$ its semi-simple part and by $\rT_\mathrm{u}$ its unipotent part, we set $\rN=\itwopi\log\rT_\mathrm{u}$ and we choose a logarithm $\rD_\mathrm{s}=\itwopi\log\rT_\mathrm{s}$. By a filtration of $(V,\rT)$ we will mean an exhaustive decreasing filtration $F^\cbbullet V$ of $V$ indexed by $\ZZ$ which is stable by $\rT_\mathrm{s}$ and such that $\rN(F^pV)\subset F^{p-1}V$ for each $p\in\ZZ$. In particular, $(V,\rT,F^\cbbullet V)$ decomposes according to the eigenvalue decomposition of $(V,\rT_\mathrm{s})$.

We also call $(V,\rT)$ the associated Betti structure of $(V,\rT,F^\cbbullet V)$ and, if~$\kk$ is a subfield of $\CC$ (\eg $\kk=\QQ$, $\RR$ or $\CC$), we say that the Betti structure is defined over $\kk$ if $(V,\rT)=\CC\otimes_\kk(V_\kk,\rT_\kk)$. The objects $(V_\kk,\rT_\kk,F^\cbbullet V)$ obviously form a category, with a duality functor and a tensor product.

We associate to these data a connection with a pole of order two, that is, the free $\CC[\hb]$-module $R_FV\defin\bigoplus_pF^pV\hb^{-p}$ equipped with the connection $\nabla=\rd+\nobreak(\rD_\mathrm{s}+\nobreak\rN/\hb)\rd\hb/\hb$. This connection has a pole of order two at $\hb=0$, but has regular singularity there, and the monodromy is given by $(V,\rT)$ (choose a semi-simple endomorphism $\rH$ with half-integral eigenvalues which commutes with $\rD_\mathrm{s}$ and such that $[\rH,\rN]=\rN$, and apply the base change~$\hb^\rH$ after a possible ramification of order two).

In the following, we will restrict to the case where \textit{the eigenvalues of $\rT$ have absolute value equal to~$1$}.

\paragraph{Connections with a pole of order two (regular singularity case)}\label{subsec:Brieskornregular}
By a \textit{connection with a pole of order two} $(\cH,\nabla)$ we mean a free $\CC\{\hb\}$-module $\cH$ of finite rank equipped with a connection $\nabla$ having a \textit{pole of order at most two}.

We set $\cG=\CC\lap\hb\rap\otimes_{\CC\{\hb\}}\cH$ (where $\CC\lap\hb\rap$ denotes the field $\CC\{\hb\}[\hbm]$ of convergent Laurent series in the variable $\hb$), equipped with the induced connection $\nabla$. \textit{We assume in this subsection that $(\cG,\nabla)$ has a regular singularity}. There exists then a canonical decreasing filtration $V^\cbbullet\cG$ indexed by $\RR$, but with a finite number of jumps modulo~$\ZZ$, such that each $V^a\cG$ is a free $\CC\{\hb\}$-submodule of $\cG$ of rank equal to $\dim_{\CC\lap\hb\rap}\cG$, on which the connection has a pole of order at most one whose residue has eigenvalues with real part in $[a,a+1)$, and such that $\hb^kV^a\cG=V^{a+k}\cG$ for each $k\in\ZZ$ and each $a\in\RR$. In the following, \textit{we will assume that the eigenvalues of the residue are real}. We will also denote by $(\cH,\nabla)$ the unique extension of the germ $(\cH,\nabla)$ as a free $\cO_{U_0}$-module (\cf Notation \ref{notation}) equipped with a meromorphic connection having a pole of order two at the origin and no other pole. If needed, we will use the unique $\CC[\hb]$-submodule of $\Gamma(U_0,\cH)$ (Birkhoff extension) on which the connection has a regular singularity at infinity (by algebraizing the Deligne meromorphic extension of $(\cH,\nabla)$ at infinity). Recall that the functor $\CC\lap\hb\rap\otimes_{\CC[\hb,\hbm]}$ (\resp $\CC\{\hb\}\otimes_{\CC[\hb]}$) induces an equivalence between the category of free $\CC[\hb,\hbm]$-modules (\resp free $\CC[\hb]$-modules) with a connection $\nabla$ having poles at $0$ and $\infty$ only, the pole at $\infty$ being a regular singularity, and $\CC\lap\hb\rap$-vector spaces (\resp free $\CC\{\hb\}$-modules) with a connection. We will implicitly use this functor.

\num{The Betti structure}
By using the $\cO_{U_0}$-module approach, we consider the local system $\cL$ on $\CC^*$ defined as $\ker\nabla_{|\CC^*}$, that we call the \textit{Betti structure of the connection with a pole of order two}. Our previous assumption amounts to assuming that the eigenvalues of the monodromy of $\cL$ have absolute value equal to~$1$. A $\kk$-structure of the connection with a pole of order two $(\cH,\nabla)$ consists by definition of a $\kk$-structure of the corresponding local system $\cL$. The objects $(\cH,\nabla,\cL_\kk)$ of connections with a pole of order two with a $\kk$-Betti structure obviously form a category with a duality functor and a tensor product.

\num{$V$-gradation}
The gradation functor replaces $(\cG,\nabla)$ with $(\gr_V\cG,\gr_V\nabla)$, where we have set $\gr_V\cG=\bigoplus_{a\in\RR}(V^a\cG/V^{>a}\cG)$ and where $\gr_V\nabla$ is the naturally induced connection: for a fixed $a\in\RR$, $\bigoplus_{k\in\ZZ}(V^{a+k}\cG/V^{>(a+k)}\cG)$ is isomorphic to $\CC[\hb,\hbm]\otimes_\CC(V^a\cG/V^{>a}\cG)$ and the connection is defined as $\rd+(\Res_a\nabla)\rd\hb/\hb$, with $\Res_a\nabla=a\id+\rN_a$ with $\rN_a$ nilpotent. 

That $\nabla$ has a regular singularity implies that $(\cG,\nabla)\simeq(\gr_V\cG,\gr_V\nabla)$. The connection with a pole of order two induces a graded connection with a pole of order two $(\gr_V\cH,\nabla)$ in a natural way, although it is not isomorphic to $(\cH,\nabla)$ in general.

\begin{lemme}
The functor $(\cH,\nabla)\mto(V,\rT,F^\cbbullet V)$, where
\begin{gather*}
V=\tbigoplus_{a\in(-1,0]}\gr_V^a\cG,\qquad
\rT=\tbigoplus_{a\in(-1,0]}\exp(-2\pi i\Res_a\nabla),\\
F^pV=\tbigoplus_{a\in(-1,0]}\gr_V^a(\hb^p\cH),
\end{gather*}
induces an equivalence between the subcategory of graded connections with a pole of order two and that of filtered vector spaces with an automorphism as in \S\ref{subsec:filteredspaces}. This functor extends naturally to objects with a $\kk$-structure.
\end{lemme}

\begin{proof}
For the first part, \cf \eg \cite[Chap\ptbl III]{Bibi00}. Let us make precise that the gradation functor is the identity at the Betti level. This will imply that the $\kk$-structure on $(\cG,\nabla)$ induces a $\kk$-structure on $(\gr_V\cG,\gr_V\nabla)$. Recall that the functor of taking global multi-valued sections induces an equivalence between the category of local systems~$\cL$ to the category of vector spaces $(L,\rT)$ with an automorphism. The composed functor $(\cG,\nabla)\mto\cL\mto(L,\rT)$ is identified with
\[
(\cG,\nabla)\mto\ker[\nabla_{\partial_\hb}:\wt\cO\otimes_{\CC\lap\hb\rap}\cG\ra\wt\cO\otimes_{\CC\lap\hb\rap}\cG],
\]
where $\wt\cO$ denotes the space of germs at the origin of multi-valued holomorphic functions on $\CC^*$, equipped with its natural monodromy operator. Let $\cN_{-a}\subset\wt\cO$ be the space of linear combinations with coefficients in $\CC\lap\hb\rap$ of the multi-valued functions $\hb^{-a}(\log\hb)^\ell/\ell!$, with its natural connection, and set $\cN=\bigoplus_{a\in(-1,0]}\cN_{-a}$. Then on the one hand, the natural inclusion of $\ker[\nabla_{\partial_\hb}:\cN\otimes_{\CC\lap\hb\rap}\cG\to\cN\otimes_{\CC\lap\hb\rap}\cG]$ into $\ker[\nabla_{\partial_\hb}:\wt\cO\otimes_{\CC\lap\hb\rap}\cG\to\wt\cO\otimes_{\CC\lap\hb\rap}\cG]$ is an isomorphism compatible with monodromy, and on the other hand, the first term is canonically identified to $(V,\rT)$ as defined in the lemma. Since the gradation functor is the identity at the $(V,\rT)$ level, it remains the identity through the previous canonical functors at the $(L,\rT)$ level, hence at the Betti level.
\end{proof}

\paragraph{Connections with a pole of order two (nr.\,exponential case)}\label{subsec:Brieskornexp}
We now relax the condition that $(\cG,\nabla)$ has a regular singularity. Since~$\nabla$ has a pole of order at most two on $\cH$, the connection has slopes $\leq1$ (\cf \eg \cite{Malgrange91}) and we say, following \cite[Chap\ptbl XII]{Malgrange91}, that $(\cG,\nabla)$ has \textit{exponential type}. We will moreover assume that the slopes are either~$0$ or~$1$. This is equivalent to the property that  the Levelt-Turrittin decomposition of $(\wh\cG,\wh\nabla)\defin\CC\lpr\hb\rpr\otimes_{\CC\lap\hb\rap}(\cG,\nabla)$ needs \textit{no ramification}, that is, $(\wh\cG,\wh\nabla)$ is isomorphic to a finite sum indexed by a finite subset $C\subset\CC$:
\begin{equation}\label{eq:formalization}
(\wh\cG,\wh\nabla)\simeq\tbigoplus_{c\in C}(\cE^{-c/\hb}\otimes\wh\cG_c),
\end{equation}
where $(\wh\cG_c,\nabla_c)$ is a connection with regular singularities and $\cE^{-c/\hb}\otimes\wh\cG_c$ is the $\CC\lpr\hb\rpr$-vector space $\wh\cG_c$ equipped with the connection $\nabla_c+c\id \rd\hb/\hb^2$. It is known (\cf\eg \cite[Rem\ptbl II.5.8]{Bibi00}) that, in such a case, $(\wh\cH,\wh\nabla)$ decomposes accordingly, as
\begin{equation}\label{eq:formalizationH}
(\wh\cH,\wh\nabla)\defin\CC\lcr\hb\rcr\otimes_{\CC\{\hb\}}(\cH,\nabla)\simeq\tbigoplus_{c\in C}(\cE^{-c/\hb}\otimes\wh\cH_c),
\end{equation}
where $(\cH_c,\nabla_c)$ is a connection with a pole of order two with regular singularities as in \S\ref{subsec:Brieskornregular}. All over this article, \textit{\nrexp} will be a shortcut for \textit{exponential type with no ramification}.

\begin{exemple}
If the leading matrix of $\nabla$ in some $\CC\{\hb\}$-basis of $\cH$ is semi-simple, then $(\cH,\nabla)$ has \nrexp. On the other hand, consider the connection with matrix
\[
A(\hb)\rd\hb:=P(\hb)\Big(\frac Y\hb+\id\Big)P(\hb)^{-1}\cdot\frac{\rd\hb}\hb,
\]
with
\[\arraycolsep3.5pt
Y=\begin{pmatrix}0&0&0\\1&0&0\\0&1&0\end{pmatrix},\quad
P(\hb)=\id+\hb Z,\quad
Z=\begin{pmatrix}0&0&1\\0&0&0\\0&0&0\end{pmatrix}.
\]
Then the matrix $\hb^2A(\hb)=P(Y+\hb\id)P^{-1}$ has characteristic polynomial equal to~$\chi(\lambda)=(\lambda-\nobreak\hb)^3$, but one can check that it is of exponential type but needs ramification (compare with \cite[Rem\ptbl2.13]{K-K-P08}).
\end{exemple}

\num{The Betti structure (Stokes filtration)}\label{num:Bettik}
The local system $\cL$ attached to $(\cG,\nabla)$ comes equipped with a family of pairs of subsheaves $\cL_{<c}\subset\cL_{\leq c}$ for each $c\in\CC$, which satisfies the properties below (\cf \cite{Deligne78}, \cite{Malgrange91}, \cite{H-S09}, \cite[Lect\ptbl2]{Bibi10}). For a fixed $\hb\in\CC^*$, define a partial order $\leqhb$ on $\CC$ compatible with addition by setting $c\leqhb0$ iff $c=0$ or $\reel(c/\hb)<0$ (and $c\lehb0$ iff $c\neq0$ and $\reel(c/\hb)<0$). This partial order on $\CC$ only depends on $\hb/|\hb|$. The required properties are as follows.
\begin{itemize}
\item
For each $\hb\in\CC^*$, the germs $\cL_{\leq c,\hb}$ form an exhaustive increasing filtration of $\cL_\hb$, compatible with the order $\leqhb$.
\item
For each $\hb\in\CC^*$, the germ $\cL_{<c,\hb}$ can be recovered as $\sum_{c'\lehb c}\cL_{\leq c',\hb}$.
\item
The graded sheaves $\cL_{\leq c}/\cL_{<c}$ are local systems on $\CC^*$.
\item
The rank of $\bigoplus_{c\in\CC}\cL_{\leq c}/\cL_{<c}$ is equal to the rank of $\cL$, so that both local systems are locally isomorphic, and there is only a finite set $C\subset\CC$ of jumping indices.
\end{itemize}
We will say that $(\cL,\cL_\bbullet)$ is a \textit{Stokes-filtered local system of \nrexp}, or simply a \textit{Stokes-filtered local system}, as we will only consider those of \nrexp in this article. A \textit{$\kk$-structure} consists of a Stokes-filtered local system $(\cL_\kk,\cL_{\kk,\bbullet})$ such that $(\cL,\cL_\bbullet)=\CC\otimes_\kk(\cL_\kk,\cL_{\kk,\bbullet})$ (\cf \cite[Def\ptbl2.14]{K-K-P08} and \cite[Prop\ptbl2.26]{Bibi10} for an equivalent definition). In particular, the monodromy of $\cL$ is defined over $\kk$ and, if $\kk\subset\RR$, this implies that $\reel\Tr(\Res\nabla)\in\hZZ$. On the other hand, recall that the Riemann-Hilbert functor $(\cG,\nabla)\mto(\cL,\cL_\bbullet)$ is an equivalence of categories (\cf \cite{Deligne78} or \cite[p\ptbl58]{Malgrange91}) compatible with duality and tensor product.

The decomposition \eqref{eq:formalization} or \eqref{eq:formalizationH} is unique, and the formalization functor $(\cG,\nabla)$ corresponds, via the Riemann-Hilbert functor $(\cG,\nabla)\mto(\cL,\cL_\bbullet)$ to the Stokes grading functor $(\cL,\cL_\bbullet)\mto\bigoplus_{c\in\CC}\cL_{\leq c}/\cL_{<c}$, so that the local system associated to $\cG_c$ is $\gr_c\cL$ (\cf \loccit). As a consequence, a $\kk$\nobreakdash-structure on $(\cH,\nabla)$ induces a $\kk$-structure on each $(\cH_c,\nabla_c)$.

\num{The Betti structure (Stokes data)}\label{num:BettiStokesdata}
The previous description of the Betti structure is independent of any choice. On the other hand, the description with Stokes data below depends on some choices (\cf \eg \cite{H-S09} for details). Let~$C$ be a non-empty finite subset of $\CC$. We say that $\theta_o\in\RR/2\pi\ZZ$ is \textit{generic} with respect to $C$ if the set $C$ is totally ordered with respect to $\leqhb$ when $\hb=e^{i\theta_o}$. Once~$\theta_o$ generic with respect to~$C$ is chosen, there is a unique numbering $\{c_1,\dots,c_n\}$ of the set~$C$ in strictly increasing order. We will set $\theta'_o=\theta_o+\pi$. Note that the order is exactly reversed at~$\theta'_o$, so that $-C$ is numbered as $\{-c_1,\dots,-c_n\}$ by $\theta'_o$.

The category of Stokes data of type $(C,\theta_o)$ defined over $\kk$ has objects consisting of two families of $\kk$-vector spaces $(L_{c,1},L_{c,2})_{c\in C}$ and a diagram of morphisms
\begin{equation}\label{eq:catStokesdata}
\begin{array}{c}
\xymatrix@C=1.5cm{
L_1=\tbigoplus_{i=1}^nL_{c_i,1}\ar@<-2mm>@/^1.7pc/[r]^-{S}\ar@<2mm>@/_1.7pc/[r]_-{S'}&L_2=\tbigoplus_{i=1}^nL_{c_i,2}
}
\end{array}
\end{equation}
such that
\begin{enumerate}
\item
$S=(S_{ij})_{i,j=1,\dots, n}$ is block-upper triangular, \ie $S_{ij}:L_{c_i,1}\to L_{c_j,2}$ is zero unless $i\leq j$, and $S_{ii}$ is invertible (so $\dim L_{c_i,1}=\dim L_{c_i,2}$, and $S$ itself is invertible),
\item
$S'=(S'_{ij})_{i,j=1,\dots, n}$ is block-lower triangular, \ie $S'_{ij}:L_{c_i,1}\to L_{c_j,2}$ is zero unless $i\geq j$, and $S'_{ii}$ is invertible (so $S'$ itself is invertible).
\end{enumerate}

A morphism of Stokes data of type $(C,\theta_o)$ consists of morphisms of $\kk$\nobreakdash-vector spaces $\lambda_{c,\ell}:L_{c,\ell}\to L'_{c,\ell}$, $c\in C$, $\ell=1,2$, which are compatible with the corresponding diagrams \eqref{eq:catStokesdata}. This allows one to classify Stokes data of type $(C,\theta_o)$ up to isomorphism. The monodromy $\rT_1$ on $L_1$ is defined by $\rT_1=S^{-1}S'$. \textit{Grading} the Stokes data means replacing $(S,S')$ with their block diagonal parts. There is a natural notion of tensor product in the category of Stokes data of type $(C,\theta_o)$, and a duality from Stokes data of type $(C,\theta_o)$ to Stokes data of type $(-C,\theta_o)$.

Fixing bases in the spaces $L_{c,\ell}$, $c\in C$, $\ell=1,2$, allows one to present Stokes data by matrices $(\Sigma,\Sigma')$ where $\Sigma=(\Sigma_{ij})_{i,j=1,\dots, n}$ (\resp $\Sigma'=(\Sigma'_{ij})_{i,j=1,\dots, n}$) is block-lower (\resp -upper) triangular and each $\Sigma_{ii}$ (\resp $\Sigma'_{ii}$) is invertible. The matrix $\Sigma_{ii}^{-1}\Sigma'_{ii}$ is the matrix of monodromy of $L_{c_i,1}$, while $\Sigma^{-1}\Sigma'$ is that of the monodromy of $L_1$.

Given $\theta_o$ generic with respect to $C$, there is an equivalence (depending on $\theta_o$) between the category of Stokes filtered local systems $(\cL_\kk,\cL_{\kk,\bbullet})$ defined over $\kk$ with jumping indices in $C$ and that of Stokes data of type $(C,\theta_o)$ defined over $\kk$, which is compatible with grading, duality and tensor product (\cf\eg\cite[\S2]{H-S09}).

\paragraph{Connections with a pole of order two of \nrexp obtained by Laplace transformation}\label{subsec:Laplace}\mbox{}%

\num{Inverse Laplace transformation}\label{num:Laplace}
The condition for a meromorphic connection $(\cG,\nabla)$ to have only slopes $0$ and $1$ is equivalent to the property that $(\cG,\nabla)$ is obtained by a Laplace transformation procedure from a meromorphic connection with regular singularity. Let us make this precise. Recall that, by extending $(\cG,\nabla)$ as a free $\cO_\CC(*0)$-module with connection having a pole at $\hb=0$ only, then choosing the Deligne extension as a $\cO_{\PP^1}(*\{0,\infty\})$-module on which the connection has a regular singularity at infinity, and then taking global sections, we find a free $\CC[\hb,\hbm]$-module with connection $(G,\nabla)$. Then $G$ is a left $\CC[\hb']\langle\partial_{\hb'}\rangle$-module ($\hb'\defin\hbm$). If we identify the ring $\CC[\hb']\langle\partial_{\hb'}\rangle$ to $\CC[t]\langle\partial_t\rangle$ by the isomorphism $\hb'=\partial_t$, $\partial_{\hb'}=-t$, then $G$ is a $\CC[t]\langle\partial_t\rangle$-module that we denote by~$\Foub G$. The condition that $(\cG,\nabla)$ is of \nrexp is equivalent to the condition that~$\Foub G$ has only \textit{regular singularities}, at finite distance and at infinity, on the complex $t$-line (\cf\eg\cite[Lemma 1.5]{Bibi08}). We call $\Foub G$ the inverse Laplace transform of $(G,\nabla)$. Equivalently, $(G,\nabla)$ is the Laplace transform of $\Foub G$ with kernel~$e^{-t\hb'}$, and we use the notation $G=\Fou(\Foub G)$. Note that the action of $\partial_t$ on~$\Foub G$ is bijective.

\num{Minimal extension and Brieskorn lattice}\label{num:minextBr}
We will have to consider the \textit{minimal extension} $M$ of~$\Foub G$. By definition, this is the unique $\Clt$-submodule of $\Foub G$ such that~$M$ has neither sub- nor quotient $\Clt$-module supported on a point. The Laplace transform $\Fou M$ of $M$ satisfies in turn $\CC[\hb',\hb^{\prime-1}]\otimes_{\CC[\hb']}\Fou M=G$.

Any good filtration of $M$ (in the sense of $\CC[t]\langle\partial_t\rangle$-modules) gives rise to a vector bundle $(\cH,\nabla)$ in $(\cG,\nabla)$ whose connection has a pole of order two. We will not recall its construction here, \cf \cite[\S1.d]{Bibi05}, \cite[\S1.c]{Bibi08} for details. It is called the \textit{Brieskorn lattice of the good filtration}.

\num{Betti structure}\label{num:LaplaceBetti}
We can pass from the Betti structure of $M$ to that of~$\Fou M$ by the \textit{topological Laplace transformation}, and the way back by the inverse topological Laplace transformation (\cf\cite{Malgrange91}).

Let $\ccV$ be the local system of horizontal sections of $M$ away from its singularities $C\subset\Afu$. Assume that it is defined over $\kk$, that is, $\ccV=\CC\otimes_\kk\ccV_\kk$. Let $j:X=\Afu\moins C\hto\Afu$ denote the inclusion. The analytic de~Rham complex $\DR^\an M$ on $\Afu$ has cohomology in degree zero only, equal to $j_*\ccV$, hence has a natural $\kk$-structure given by $j_*\ccV_\kk$.

The topological Laplace transform of the $\kk$-perverse sheaf $\cF\defin j_*\ccV_\kk[1]$ is defined in \cite[Chap\ptbl VI,\,\S2]{Malgrange91} (\cf also \cite[\S7.d]{Bibi10}, and \cite{K-K-P08} for a different approach), as a perverse sheaf on $\Afu_{\hb'}$ with a Stokes structure at infinity, that we denote by $(\FcF,\FcF_\bbullet)$. Forgetting the behaviour of this object near $\hb'=0$ (this corresponds to tensoring $\Fou M$ with $\CC[\hb',\hb^{\prime-1}]$, that is, to considering $G$) allows one to describe it as a Stokes-filtered local system $(\cL_\kk,\cL_{\kk,\bbullet})$. One could also use classical integral formulas for the Stokes matrices (\cf \cite{B-J-L81}) to describe the $\kk$-structure.

Notice that in \cite[Prop\ptbl4.7]{H-S09} one finds conversely the description of $\ccV_\kk$ in terms of $(\cL_\kk,\cL_{\kk,\bbullet})$.

Let now $Q_\rB$ be a nondegenerate pairing $\ccV_\kk\otimes\ccV_\kk\to\kk_X$. It extends as a nondegenerate pairing $j_*Q_\rB:j_*\ccV_\kk\otimes j_*\ccV_\kk\to\kk_{\Afu}$. By topological Laplace transformation, we get a pairing $\wh{j_*Q_\rB}:(\cL_\kk,\cL_{\kk,\bbullet})\otimes\iota^{-1}(\cL_\kk,\cL_{\kk,\bbullet})\to\kk_\bbS$. Its germ on $\cL_\kk\otimes\iota^{-1}\cL_\kk$ (that is, forgetting the Stokes filtration) at $\hbo\in\bbS$ is described as follows. Let $\wt\PP^1$ be oriented real blow-up space of $\PP^1$ at $t=\infty$. This is topologically a disc, obtained by adding a boundary $S^1$ to $\Afu$, with coordinate $(\infty,e^{i\theta})$. For each $\hbo\in\bbS$, let $\Phi_{\pm\hbo}$ denote the family of closed sets of $\Afu$ whose closure in $\wt\PP^1$ does not cut the closed set $\{(\infty,e^{i\theta})\mid\reel(e^{i\theta}/\pm\hbo)\geq0\}$. Then $\wh{j_*Q_\rB}_{\hbo}$ is the pairing induced by the cup product followed by $Q_\rB$:
\[
H^1_{\Phi_{\hbo}}(\Afu,j_*\ccV_\kk)\otimes H^1_{\Phi_{-\hbo}}(\Afu,j_*\ccV_\kk)\to H^2_\rc(\Afu,\QQ)\simeq\QQ,
\]
where we remark that the intersection family $\Phi_{\hbo}\cap \Phi_{-\hbo}$ is that of compact sets in~$\Afu$.

\begin{proposition}\label{prop:whQ}
The Laplace transform $\wh{j_*Q_\rB}$ induces a nondegenerate pairing $(\cL_\kk,\cL_{\kk,\bbullet})\otimes\iota^{-1}(\cL_\kk,\cL_{\kk,\bbullet})\to\kk_\bbS$.
\end{proposition}

\begin{proof}[Sketch of proof]
This follows from Poincaré duality if we forget the Stokes filtration, and we have to check that the Stokes filtration behaves correctly. Let us denote by~$\bD$ the Poincaré-Verdier duality functor. Then one shows for any $c\in\CC$, the existence of isomorphisms
\begin{equation}\label{eq:DF}
\Fou(\bD\cF)_{<c}\simeq\iota^{-1}\bD(\FcF_{\leq c}),\quad\Fou(\bD\cF)_{\leq c}\simeq\iota^{-1}\bD(\FcF_{<c})
\end{equation}
which are exchanged by duality up to bi-duality isomorphisms. In order to do so, one uses an integral formula for the topological Laplace transformation (\cf\cite[\S VI.2]{Malgrange91}, see also \cite[Lect\ptbl7]{Bibi10}). These isomorphisms are obtained by a local duality statement on the two-dimensional space of variables $t,z'$ suitably blown-up, and then by using the commutation (up to a suitable shift) of Poincaré-Verdier duality with smooth pull-back and proper push-forward which enter in the definition of the topological Laplace transformation.
\end{proof}

\paragraph{Deligne-Malgrange lattices}\label{subsec:DML}
We explain here another example of pairs $(\cH,\nabla)$ of \nrexp. It can be obtained as the Brieskorn lattice of a natural filtration of $M$, namely the filtration by Deligne lattices.

Let $(\cG,\nabla)$ be a meromorphic connection of \nrexp, that is, satisfying \eqref{eq:formalization}. The functor which associates to any lattice $\cH$ of $\cG$ (\ie a $\CC\{\hb\}$-free submodule such that $\CC\lap\hb\rap\otimes_{\CC\{\hb\}}\cH=\cG$) its formalization $\CC\lcr\hb\rcr\otimes_{\CC\{\hb\}}\cH$ is an equivalence between the full subcategory of lattices of $\cG$ and that of lattices of $\wh\cG$ (\cf \cite{Malgrange95}). In particular, let us consider for each $c\in C$ and $a\in\RR$ the Deligne lattices $\wh\cG_c{}^a$ (\resp $\wh\cG_c{}^{>a}$) of the regular connection $(\wh\cG_c,\wh\nabla_c)$ considered in \eqref{eq:formalization}, characterized by the property that the connection $\wh\nabla_c$ on $\wh\cG_c^a$ has a simple pole and the real parts of the eigenvalues of its residue belong to $[a,a+1)$ (\resp to $(a,a+1]$). Clearly, $\wh\cG_c{}^{a+1}=\hb\wh\cG_c{}^a$ and $\wh\cG_c{}^{>a+1}=\hb\wh\cG_c{}^{>a}$.

According to the previous equivalence, there exist unique lattices of $\cG$, denoted by $\DM^a(\cG,\nabla)$ (\resp $\DM^{>a}(\cG,\nabla)$, which induce, by formalization, the decomposed lattice $\bigoplus_c\cE^{-c/\hb}\otimes\wh\cG_c{}^a$ (\resp $\bigoplus_c\cE^{-c/\hb}\otimes\wh\cG_c{}^{>a}$). They are called the Deligne-Malgrange lattices of $(\cG,\nabla)$. We regard them as defining a decreasing filtration of $\cG$.

\begin{lemme}
Any morphism $(\cG,\nabla)\to(\cG',\nabla')$ of meromorphic connections of \nrexp is strictly compatible with the filtration by Deligne-Malgrange lattices.
\end{lemme}

\begin{proof}[Sketch of proof]
The associated formal morphism is block-diagonal with respect to the decomposition \eqref{eq:formalization}, and each diagonal block induces a morphism between the corresponding regular parts, which is known to be strict with respect to the filtration by the Deligne lattices.
\end{proof}

The behaviour by duality below is proved similarly by reducing to the regular singularity case (\cf \eg \cite[\S III.1.b]{Bibi00} or \cite[Lem\ptbl3.2]{Bibi96bb}).

\begin{lemme}\label{lem:dualDM}
Let $(\cG,\nabla)$ be as above and let $(\cG,\nabla)^\vee$ be the dual meromorphic connection. Then, there are canonical isomorphisms
\tagdroite
\begin{align*}
[\DM^a(\cG,\nabla)]^\vee&\simeq\DM^{>-a-1}[(\cG,\nabla)^\vee],\\
[\DM^{>a}(\cG,\nabla)]^\vee&\simeq\DM^{-a-1}[(\cG,\nabla)^\vee].\tag*{\qed}
\end{align*}
\taggauche
\end{lemme}

We will use this lemma as follows ($\iota^*$, Notation \ref{notation},  will be needed later).

\begin{corollaire}\label{cor:pV}
Let $(\cG,\nabla)$ be of \nrexp, with associated Stokes structure $(\cL,\cL_\bbullet)$. Let $\ccQ_\rB:(\cL,\cL_\bbullet)\otimes_\CC\iota^{-1}(\cL,\cL_\bbullet)\to\CC$ be a nondegenerate bilinear pairing. Let $\ccQ:(\cG,\nabla)\otimes_{\CC\lap\hb\rap}\iota^*(\cG,\nabla)\to(\CC\lap\hb\rap,\rd)$ be the nondegenerate pairing corresponding to $\ccQ_\rB$ via the Riemann-Hilbert correspondence. Then, for each $a\in\RR$, $\ccQ$ extends in a unique way as a nondegenerate pairing $\DM^a(\cG,\nabla)\otimes_{\CC\{\hb\}}\iota^*\DM^{>-a-1}(\cG,\nabla)\to(\CC\{\hb\},\rd)$.\qed 
\end{corollaire}

\begin{corollaire}\label{cor:DMA}
With the assumptions of Corollary \ref{cor:pV}, assume moreover that~$a$ is an integer. Then,
\begin{enumerate}
\item\label{cor:DMA1}
if none of the monodromies of the $\wh\cG_c$ has $1$ as an eigenvalue, then $\DM^a(\cG,\nabla)=\DM^{>a}(\cG,\nabla)$ for each integer $a$, and $\ccQ$ induces a nondegenerate pairing $\DM^a(\cG,\nabla)\otimes_{\CC\{\hb\}}\iota^*\DM^a(\cG,\nabla)\to(\hb^{2a+1}\CC\{\hb\},\rd)$,
\item\label{cor:DMA2}
if none of the monodromies of the $\wh\cG_c$ has $-1$ as an eigenvalue, then $\DM^{a-1/2}(\cG,\nabla)=\DM^{>a-1/2}(\cG,\nabla)$ for each integer $a$, and $\ccQ$ induces a nondegenerate pairing $\DM^{a-1/2}(\cG,\nabla)\otimes_{\CC\{\hb\}}\iota^*\DM^{a-1/2}(\cG,\nabla)\to(\hb^{2a}\CC\{\hb\},\rd)$.\qed
\end{enumerate}
\end{corollaire}

\section{Non-commutative Hodge structures}\label{sec:ncHodge}

In this section, we fix a subfield $\kk$ of $\RR$, \eg $\kk=\QQ$ or $\RR$. The presentation given below owes much to various sources: \cite{Hertling01,Hertling06}, \cite{H-S06,H-S08b}, \cite{K-K-P08}, \cite{Mochizuki07,Mochizuki08b} and~\cite{Bibi01c}.

\paragraph{Non-commutative Hodge structures via linear algebra}\label{subsec:linalg}

\num{A reminder on Hodge structures}\label{num:reminder}
Let $H$ be a finite dimensional complex vector space. Recall that a complex Hodge structure of weight $w\in\nobreak\ZZ$ consists in a decomposition $H=\bigoplus_{p\in\ZZ}H^{p,w-p}$. Equivalently, it consists of a semi-simple endomorphism~$\cQ$ of $H$ with half-integral eigenvalues. The eigenspace of $\cQ$ corresponding to the eigenvalue $p-w/2$, $p\in\ZZ$, is $H^{p,w-p}$. The role of the weight only consists in fixing the bigrading.

A real structure is a $\RR$-vector space $H_\RR$ such that $H=\CC\otimes_\RR H_\RR$, with respect to which $H^{w-p,p}=\ov{H^{p,w-p}}$. Then the matrix of $\cQ$ in any basis of $H_\RR$ is purely imaginary.

On the other hand, a polarization of a complex Hodge structure is a nondegenerate $(-1)^w$-Hermitian pairing $k$ on $H$ such that the decomposition is $k$\nobreakdash-orthogonal and such that the Hermitian form $h$ on $H$ defined by $h_{|H^{p,w-p}}=i^{p-(w-p)}k_{|H^{p,w-p}}=i^{-w}(-1)^pk_{|H^{p,w-p}}$ is positive definite, in other words, defining the Weil operator $C$ by $e^{\pi i\cQ}$, $h=k(C\cbbullet,\ov\cbbullet)$. For a real Hodge structure, the real polarization $Q$ is then defined as the real part of~$k$, and it is $(-1)^w$-symmetric.

\num{Complex \ncHodge structures}\label{num:CncH}
By a \textit{complex \ncHodge structure} of weight $w\in\nobreak\ZZ$, we mean the data $(H,\cU,\cU^\dag,\cQ,w)$, where $\cU,\cU^\dag,\cQ$ are endomorphisms of $H$. When~$w$ is fixed, these data form a category, where morphisms are linear morphisms $H\to H'$ commuting with the endomorphisms $\cU,\cU^\dag,\cQ$. For a complex Hodge structure, we have $\cU=\cU^\dag=0$ and $\cQ$ is as above. The category of complex \ncHodge structures of weight $w\in\ZZ$ is abelian.

\begin{exemple*}
Assume $\cU=\cU^\dag=0$ and $\cQ$ is semi-simple. One can decompose $H=\bigoplus_{\lambda\in\CC^*}H_\lambda$, where $H_\lambda$ is the $\lambda$-eigenspace of $e^{-2\pi i\cQ}$. Then each $(H_\lambda,0,0,\cQ,w)$ is a Hodge structure of weight~$w$ and we can regard $(H,0,0,\cQ,w)$ as a Hodge structure of weight~$w$ equipped with a semi-simple automorphism, with eigenvalue $\lambda$ on $H_\lambda$.
\end{exemple*}

In order to understand various operations on complex \ncHodge structures, we associate to $(H,\cU,\cU^\dag,\cQ,w)$ the $\CC[\hb]$-module $\cH=\CC[\hb]\otimes_\CC H$, with the connection
\bgroup\numstareq
\begin{equation}\label{eq:assconn}
\nabla=\rd+\big(\hbm \cU-(\cQ+(w/2)\id)-\hb \cU^\dag\big)\frac{\rd\hb}{\hb}.
\end{equation}
\egroup
This connection has a (possibly irregular) singularity at $\hb=0$ and $\hb=\infty$, and no other singularity. Duality and tensor product are defined in a natural way, according to the rules for connections. Hence
\[
(H,\cU,\cU^\dag,\cQ,w)^\vee=(H^\vee,-{}^t\!\cU,-{}^t\!\cU^\dag,-{}^t\!\cQ,-w),
\]
and $(H_1,\cU_1,\cU_1^\dag,\cQ_1,w_1)\otimes(H_2,\cU_2,\cU_2^\dag,\cQ_2,w_2)$ has weight $w_1+w_2$ and the endomorphisms are defined by formulas like $\cU_1\otimes\id_2+\id_1\otimes \cU_2$. The involution $\iota:\hb\mto-\hb$ induces a functor $\iota^*$, with $\iota^*(H,\cU,\cU^\dag,\cQ,w)=(H,-\cU,-\cU^\dag,\cQ,w)$.

\num{Real \ncHodge structures}
The complex conjugate of the complex \ncHodge structure is defined as
\[
\ov{(H,\cU,\cU^\dag,\cQ,w)}\defin(\ov H,\ov \cU{}^\dag,\ov \cU,-\ov \cQ,w), 
\]
where $\ov H$ is the $\RR$-vector space $H$ together with the conjugate complex structure. A real structure $\kappa$ on $(H,\cU,\cU^\dag,\cQ,w)$ is an isomorphism from it to its conjugate, such that $\ov\kappa\circ\kappa=\id$. A real structure consists therefore in giving a real structure $H_\RR$ on~$H$, with \hbox{respect} to which $\cU^\dag=\ov \cU$ and $\ov \cQ+\cQ=0$. We denote such a structure as $(H_\RR,\cU,\cQ,w)$. Morphisms are $\RR$-linear morphisms compatible with $\cU$ and~$\cQ$. Real \ncHodge structures $(H_\RR,\cU,\cQ,w)$ satisfy properties similar to that of complex \ncHodge structures and we have similar operations defined in a natural way.

\num{Polarization of a complex \ncHodge structure}
A \textit{polarization} of $(H,\cU,\cU^\dag,\cQ,w)$ is a nondegenerate Hermitian form $h$ on $H$ such that
\begin{itemize}
\item
$h$ is positive definite,
\item
$\cU^\dag$ is the $h$-adjoint of $\cU$ and $\cQ$ is self-adjoint with respect to $h$.
\end{itemize}
It is useful here to introduce the complex Tate object $\TT_\CC(\ell)$ defined as $(\CC,0,0,0,-2\ell)$ for $\ell\in\ZZ$, corresponding to the Hodge structure $\CC^{-\ell,-\ell}$. The Tate twist by $\TT_\CC(\ell)$ is simply denoted by $(\ell)$. The last condition is equivalent to asking that $h$ defines an isomorphism
\[
(H,\cU,\cU^\dag,\cQ,w)\isom\iota^*\ov{(H,\cU,\cU^\dag,\cQ,w)}{}^\vee(-w).
\]
The tensor product
\[
(H_1,\cU_1,\cU_1^\dag,\cQ_1,h_1,w_1)\otimes(H_2,\cU_2,\cU_2^\dag,\cQ_2,h_2,w_2)
\]
of polarized complex \ncHodge structures is defined by the supplementary relation $h=h_1\otimes h_2$, and is also polarized.

\num{Polarization of a real \ncHodge structure and the Betti structure}\label{num:BettincHodge}%
Although the notion of a real \ncHodge structure seems to be defined over~$\RR$, the real vector space~$H_\RR$ does not contain all the possible ``real'' information on the structure, in cases more general than that of a Hodge structure. The Weil operator is not defined in this setting. The formula $C=e^{\pi i\cQ}$ exhibits the Weil operator as a square root of the monodromy of the connection $d-\cQ \rd\hb/\hb$. This suggests that the monodromy of the connection \eqref{eq:assconn} should be taken into account in order to properly define the notion of a real \ncHodge structure, and further, that of a \nckHodge structure. Even further, if $\nabla$ has an irregular singularity, the Betti real structure is encoded in the Stokes data attached to the connection, not only in the monodromy, as explained in \S\ref{num:Bettik}, together with the notion of a $\kk$-Betti structure, for any subfield~$\kk$ of~$\RR$ (\eg $\kk=\QQ$).

Here is another drawback of the presentation of a complex \ncHodge structure as a vector space with endomorphisms: the notion of a variation of such objects is not defined in a holomorphic way, exactly as $H^{p,w-p}$ do not vary holomorphically in classical Hodge theory. Good variations are characterized by the property of the Hermitian metric to be harmonic in the sense of \cite{Simpson92}, and the endomorphisms $\cU,\cQ$ satisfy relations encoded in the notion of a \textit{CV structure} \cite{Hertling01}.

The correct generalization of the Hodge filtration is that of a vector bundle on the complex affine line together with a meromorphic connection. This motivates the definition of a \ncHodge structure in \S\ref{subsec:ncHodge} by taking integrable twistor structures (\ie vector bundles on $U_0$ with a meromorphic connection having a pole of order two at zero and no other pole, plus gluing data with the ``twistor conjugate'' $\gamma^*\ov\cH$ or ``twistor adjoint'' $\sigma^*\ov\cH{}^\vee$ object) as the starting point.

\paragraph{Non-commutative Hodge structures via integrable twistor structures}\label{subsec:ncHodge}
\num{Non-commutative $\kk$-Hodge structures}\label{num:nckHodge}
Let $(\cH,\nabla)$ be a connection with a pole of order two of rank $d$, and let $\cL$ be the corresponding local system on~$\bbS$. A real structure $\cL_\RR$ on $\cL$ allows one to produce a holomorphic vector bundle with connection $(\wt\cH,\wt\nabla)$ on $\PP^1$, by gluing $(\cH,\nabla)$ (chart~$U_0$) with $\gamma^*\ov{(\cH,\nabla)}$ (chart $U_\infty$, \cf Notation \ref{notation}) through the flat gluing isomorphism $g:(\cH,\nabla)_{|U_0\cap U_\infty}\isom\gamma^*\ov{(\cH,\nabla)}_{|U_0\cap U_\infty}$ uniquely defined as follows (\cf \cite{Hertling01}):
\begin{itemize}
\item
$g$ is uniquely defined from $g^\nabla:\cH^\nabla_{|U_0\cap U_\infty}\isom\gamma^{-1}\ov\cH{}^\nabla_{|U_0\cap U_\infty}$,
\item
$g^\nabla$ is uniquely defined from its restriction $g^\nabla_{|\bbS}:\cL\to\ov\cL$ to $\bbS$ (recall that $\gamma_{|\bbS}=\id$),
\item
$g^\nabla_{|\bbS}$ is defined to be the isomorphism induced by the real structure on $\cL$.
\end{itemize}

We say that $(\cH,\nabla,\cL_\RR)$ is \textit{pure of weight~$w$} if the bundle $\wt\cH$ is isomorphic to $\cO_{\PP^1}(w)^d$. Notice that, for $k\in\ZZ$ and for any $\lambda\in\CC$, $(\hb^k\cH,\nabla,\lambda\cL_\RR)$ is then pure of weight $w-2k$.

\begin{definition*}[{\cf \cite[Rem\ptbl2.13]{Hertling01}, \cite{K-K-P08}}]
Let $(\cH,\nabla,(\cL_\kk,\cL_{\kk,\bbullet}))$ be a connection with a pole of order two and $\kk$-Betti structure. We say that it is a \textit{pure \nckHodge structure of weight~$w$} if the underlying triple $(\cH,\nabla,\cL_\RR)$, with $\cL_\RR\defin\RR\otimes_\kk\cL_\kk$, is pure of weight~$w$ (if $\kk=\RR$  and forgetting the Stokes filtration $\cL_{\kk,\bbullet}$, we recover the notion of TER structure of \cite[Rem\ptbl2.13]{Hertling01}).
\end{definition*}

\begin{exemple*}
Let $H_\kk$ be a $\kk$-vector space and set $H=\CC\otimes_\kk H_\kk$. A~$\kk$\nobreakdash-Hodge structure of weight~$w$ consists of the data $(H_\kk,F^\cbbullet H)$ such that the filtration $F^\cbbullet H$ and its conjugate are~$w$-opposed. Set $(\cH,\nabla)=(R_FH,\rd)\defin(\bigoplus_pF^pH\hb^{-p},\rd)$ where $\rd$ is the standard differential on $H\otimes_\CC\CC[\hb,\hbm]$, that we restrict to the $\CC[\hb]$-submodule $\bigoplus_{p\in\ZZ}F^pH\hb^{-p}$ (we work here with the algebraic version of $\cH$). The local system $\cL_\RR$ is the constant local system~$H_\RR$ and $\gamma^*\ov{(\cH,\nabla)}=\bigoplus_q\ov F{}^q\hb^{\prime-q}$. The condition that $(\cH,\nabla,\cL_\RR)$ is pure of weight~$w$ is equivalent to the~$w$-opposedness property. Indeed, let us show one direction for instance. If both filtrations $F^\cbbullet$ and~$\ov F{}^\cbbullet$ are~$w$-opposed, we have a bigrading $H=\bigoplus_pH^{p,w-p}$ with $H^{q,p}=\ov H{}^{p,q}$ and~$F^\cbbullet$ defined as usual. Then we have decompositions into finite sums
\begin{align*}
R_FH&=\tbigoplus_p(H^{p,w-p}\otimes\hb^{-p}\CC[\hb]),\\
\gamma^*\ov{R_FH}&=\tbigoplus_q(\ov H{}^{q,w-q}\otimes\hb^{\prime-q}\CC[\hb'])=\tbigoplus_p(H^{p,w-p}\otimes\hb^{\prime p-w}\CC[\hb']), 
\end{align*}
and the gluing morphism is the identity $\CC[\hb,\hbm]\otimes_{\CC[\hb]}R_FH=H[\hb,\hbm]=H[\hb',\hb^{\prime-1}]=\CC[\hb',\hb^{\prime-1}]\otimes_{\CC[\hb']}\gamma^*\ov{R_FH}$. For each $p$, the line bundle defined by the gluing data $(\hb^{-p}\CC[\hb],\hb^{\prime p-w}\CC[\hb'],\id)$ is nothing but $\cO_{\PP^1}(w)$.

\end{exemple*}

\begin{exemple*}[The Tate object $\TT_\QQ(\ell)$ with $\ell\in\hZZ$]
For non-commutative Hodge structures, we can take the opportunity of having monodromy $\neq\id$ to define the Tate \nckHodge structure $\TT_\kk(\ell)$ for $\ell\in\hZZ$. We set
\bgroup\numstareq
\begin{equation}\label{eq:Tatedemi}
\TT_\kk(\ell)=(\cO_{U_0},\rd+\ell \rd\hb/\hb,(2\pi i/\hb)^\ell\kk_{\bbS}),
\end{equation}
\egroup
where $(2\pi i/\hb)^\ell\kk_{\bbS}$ denotes the rank-one local system on $\bbS$ generated by the (possibly) multivalued function $(2\pi i/\hb)^\ell$. It has monodromy equal to $(-1)^{2\ell}\id$. By the residue theorem, $\TT_\QQ(\ell)$ has weight $-2\ell$.

If $\ell\in\ZZ$, then we have an isomorphism
\bgroup\numstarstareq
\begin{equation}\label{eq:Tate}
\TT_\kk(\ell)\To{\hb^\ell}(\hb^\ell\cO_{U_0},\rd,(2\pi i)^\ell\kk_\bbS),
\end{equation}
\egroup
where we regard $\hb^\ell\cO_{U_0}$ as included in $\cO_{U_0}(1/\hb)$, with the induced differential $\rd$, and the latter object corresponds to the object constructed in the previous example for the Tate Hodge structure $\kk(\ell)$, with $H_\kk=(2\pi i)^\ell\kk$ and $H=H^{-\ell,-\ell}$.
\end{exemple*}

\num{Polarized complex \ncHodge structures}\label{num:polncHS}
Let now $(\cH,\nabla)$ be a connection with a pole of order two equipped with a nondegenerate $\iota$\nobreakdash-Hermitian pairing $\ccC_\bbS^\nabla:\cL\otimes_\CC\iota^{-1}\ov\cL\to\CC_\bbS$. Recalling that $\gamma_{|\bbS}=\id$, we regard $\ccC_\bbS^\nabla$ as a nondegenerate $\sigma$-Hermitian pairing $\cL\otimes\sigma^{-1}\ov\cL\to\nobreak\CC_\bbS$, and we extend it in a unique way as a nondegenerate pairing $\ccC^\nabla:\cH^\nabla_{|U_0\cap U_\infty}\otimes\nobreak\sigma^{-1}\ov\cH{}^\nabla_{|U_0\cap U_\infty}\to\CC_{U_0\cap U_\infty}$. This pairing in turn defines in a unique way a flat isomorphism $\ccC:\sigma^*\ov{(\cH,\nabla)}_{|U_0\cap U_\infty}\isom(\cH,\nabla)^\vee_{|U_0\cap U_\infty}$. By gluing with this isomorphism the dual bundle $\cH^\vee$ (chart $U_0$) with the $\sigma$-conjugate bundle $\sigma^*\ov{(\cH,\nabla)}$ (chart $U_\infty$), we obtain a bundle $\wh\cH$ on $\PP^1$, which has degree zero by the residue formula. Since $\ccC$ is $\sigma$-Hermitian, we obtain a natural morphism $\cS:\wh\cH\to\sigma^*\ov{\wh\cH}{}^\vee$ compatible with the connections (it is induced by $\id$ on each chart) and which is $\sigma$-Hermitian. We say that $(\cH,\nabla,\ccC)$ is \textit{a pure complex \ncHodge structure (of weight $0$)} if $\wh\cH$ is the trivial bundle on~$\PP^1$.

If $(\cH,\nabla,\ccC)$ is pure of weight $0$, the isomorphism $\cS$ induces an isomorphism on global sections. Let us set $\ov H=\Gamma(\PP^1,\wh\cH)$. Identifying in a natural way $\Gamma(\PP^1,\sigma^*\ov{\wh\cH})$ with $\Gamma(\PP^1,\ov{\wh\cH})=\nobreak H$, the isomorphism $h\defin\Gamma(\PP^1,\cS):\ov H\to\nobreak H^\vee$ is a nondegenerate Hermitian pairing $h:H\otimes\ov H\to\CC$. We then say that $(\cH,\nabla,\ccC)$ is \textit{a complex \ncHodge structure, pure (of weight~$0$) and polarized} if $h$ is positive definite.

\begin{remarque*}
The following criterion can be used for the purity and polarizability: $(\cH,\nabla,\ccC)$ is pure (of weight $0$) and polarized if and only if there exists a $\cO_{U_0}$-basis~$\epsilong$ of $\cH$ such that the matrix $\ccC(\epsilong,\sigma^*\ov\epsilong)$ is the identity (\cf \cite[Rem\ptbl2.2.3]{Bibi01c}). In this way, one checks that if $(\cH,\nabla,\ccC)$ is pure (of weight $0$) and polarized, then so is $(\hb^k\cH,\nabla,(-1)^k\ccC)$ for every $k\in\ZZ$.
\end{remarque*}

\begin{exemple*}
Let $H=\bigoplus_{p\in\ZZ}H^{p,w-p}$ be a grading of $H$, and let~$h$ be a positive definite Hermitian pairing on $H$ such that the decomposition is $h$-orthogonal. Let $k$ be the $(-1)^w$-Hermitian pairing on $H$ such that the decomposition is $k$-orthogonal and $h(\cbbullet,\ov\cbbullet)=k(C\cbbullet,\ov\cbbullet)$, where $C$ is the standard Weil operator $i^{p-q}\id$ on $H^{p,q}$. Define the Hodge filtration $F^\cbbullet H$ as usual, and $(\cH,\nabla)$ as in Example of \S\ref{num:nckHodge}. In the algebraic setting, $\cH_{|U_0\cap U_\infty}=H\otimes_\CC\CC[\hb,\hbm]$. Let $\ccC$ be defined from $\ccC^\nabla_\bbS\defin i^{-w}k$ by sesquilinearity, that~is,
\[
\ccC(\textstyle\sum_pv^p\hb^{-p},\sigma^*\ov{\sum_qw^q\hb^{-q}})\defin\sum_{p,q}i^{-w}k(v^p,\ov w{}^q)\hb^{-p}(-1/\hb)^{-q}\in\CC[\hb,\hbm].
\]
For $v^p\in H^{p,w-p}$ and $w^q\in H^{q,w-q}$, we have $\ccC(v^p\hb^{-p},\sigma^*\ov{w^q\hb^{-q}})=0$ unless $p=q$, in which case it is equal to $h(v^p,w^p)$, showing that a $h$-orthonormal basis of $\bigoplus H^{p,w-p}$ induces a  basis of $\cH$ which is $\ccC$-orthonormal in the sense of the previous remark, so $(\cH,\nabla,\ccC)$ is pure (of~weight $0$) and polarized.
\end{exemple*}

\num{Polarized \nckHodge structures}\label{num:polnckHodge}
For any integer~$w$, we regard $\hb^{-w}\cO_{U_0}$ as contained in $\cO_{U_0}[\hbm]$ and we consider on it the connection induced by $\rd$, that we still denote by~$\rd$. The corresponding local system is the constant sheaf $\CC$ on $U_0\cap U_\infty$, and we endow it with the usual $\QQ$-structure (hence $\kk$-structure). The Stokes filtration is the trivial one.

We now consider a connection with a pole of order two $(\cH,\nabla)$ with $\kk$-Betti structure $(\cL_\kk,\cL_{\kk,\bbullet})$ together with a nondegenerate $(-1)^w$-$\iota$\nobreakdash-sym\-metric pairing
\bgroup\numstareq
\begin{multline}\label{eq:polnckHodge*}
(\ccQ,\ccQ_\rB):((\cH,\nabla),(\cL_\kk,\cL_{\kk,\bbullet}))\otimes(\iota^*(\cH,\nabla),\iota^{-1}(\cL_\kk,\cL_{\kk,\bbullet}))\\
\to((\hb^{-w}\cO_{U_0},\rd),\kk_\bbS).
\end{multline}
\egroup
Here we mean that $\ccQ_\rB$ induces, by the Riemann-Hilbert correspondence, a nondegenerate $(-1)^w$-$\iota$\nobreakdash-symmetric pairing $\ccQ$ on $(\cG,\nabla)$ with values in $\cO_{U_0}(*0)$, whose restriction to $\cH$ takes values in $\hb^{-w}\cO_{U_0}$ and is non-degenerate as such (and automatically $(-1)^w$-$\iota$\nobreakdash-symmetric). We will write~$\ccQ_\rB$ instead of $(\ccQ,\ccQ_\rB)$ when this causes no confusion.

On the one hand, $(\cH,\nabla,\cL_\RR)$ defines a vector bundle $\wt\cH$. On the other hand,~$\ccQ_\rB$ induces a $(-1)^w$-$\iota$\nobreakdash-symmetric pairing $\ccQ^\nabla_\bbS:\cL_\RR\otimes_\CC\iota^{-1}\cL_\RR\to\RR$, and thus $\ccP^\nabla_\bbS\defin i^{-w}\ccQ_\bbS^\nabla$, made sesquilinear according to the isomorphism $\kappa:\cL\to\ov\cL$ induced by the real structure, defines a $\iota$\nobreakdash-Hermitian nondegenerate pairing $\ccC^\nabla_\bbS:\cL\otimes_\CC\iota^{-1}\ov\cL\to\CC$, with
\bgroup\numstarstareq
\begin{equation}\label{eq:polnckHodge**}
\ccC^\nabla_\bbS(\cbbullet,\iota^{-1}\ov\cbbullet)\defin\ccP_\bbS^\nabla(\cbbullet,\iota^{-1}\kappa(\cbbullet))=i^{-w}\ccQ_\bbS^\nabla(\cbbullet,\iota^{-1}\kappa(\cbbullet)).
\end{equation}
\egroup
Extending $\ccC^\nabla_\bbS$ as a flat $\sigma$-Hermitian pairing $\ccC$ on $\cH_{|U_0\cap U_\infty}$, we define as above a vector bundle $\wh\cH$ on $\PP^1$.

\begin{remarques*}\mbox{}
\begin{itemize}
\item
The data $(\cH,\nabla,\cL_\RR,\ccP)$ as above (forgetting the Stokes structure, if any, and setting $\kk=\RR$) is called a TERP$(-w)$-structure in~\cite{Hertling01}.
\item
Let us set $\ccS=(\hb/2\pi i)^w\ccQ$. Then we can regard $\ccS$ as a morphism of \nckHodge structures with values in $\TT_\kk(-w)$.
\end{itemize}
\end{remarques*}

\begin{lemme*}[{\cf\cite[Th\ptbl2.19]{Hertling01}}]
The pairing $\ccQ$ induces an isomorphism
\[
\wt\cH\isom\iota^*\wh\cH\otimes\cO_{\PP^1}(w).
\]
\end{lemme*}

\begin{proof}
We have an isomorphism between the bundle $\wt\cH$ defined by the gluing data $(\cH_{U_0},(\gamma^*\ov\cH)_{U_\infty},\kappa:\cH_{|U_0\cap U_\infty}\isom\gamma^*\ov\cH_{|U_0\cap U_\infty})$ and the bundle defined by the gluing data
\[
(\iota^*\cH_{U_0}^\vee\otimes\hb^{-w}\cO_{U_0},(\iota^*\sigma^*\ov\cH)_{U_\infty},\iota^*\ccC),
\]
which is given by $\ccP$ on $U_0$ and $\id$ on $U_\infty$ (since $\gamma=\sigma\circ\iota$). On the other hand, $\iota^*\wh\cH$ is defined by the gluing data $(\iota^*\cH^\vee_{U_0},(\iota^*\sigma^*\ov\cH)_{U_\infty},\iota^*\ccC)$, hence the assertion.
\end{proof}

\begin{definition*}
A connection $(\cH,\nabla)$ with a pole of order two equipped with a $\kk$-Betti structure $(\cL_\kk,\cL_{\kk,\bbullet})$ and a nondegenerate $(-1)^w$\nobreakdash-$\iota$\nobreakdash-sym\-metric $\kk$-pairing $\ccQ_\rB$ as in \eqref{eq:polnckHodge*} is called a \textit{pure and polarized \nckHodge structure of weight~$w$} if the associated triple $(\cH,\nabla,\ccC)$ is a pure complex \ncHodge structure (of~weight~$0$) which is polarized (\cf \S\ref{num:polncHS}).
\end{definition*}

A consequence of the lemma above is that, if $((\cH,\nabla),(\cL_\kk,\cL_{\kk,\bbullet}),\ccQ_\rB)$ is a pure and polarized \nckHodge structure of weight~$w$, then it is a pure \nckHodge structure of weight~$w$ (as defined in \S\ref{num:nckHodge}). Notice also that $((\hb^k\cH,\nabla),(\cL_\kk,\cL_{\kk,\bbullet}),\ccQ_\rB)$ is then pure and polarized of weight $w-2k$.

\begin{exemple*}
Let $(H_\kk,F^\cbbullet H)$ be a pure $\kk$-Hodge structure of weight~$w$ and let $Q_\rB:H_\kk\otimes H_\kk\to\kk$ be a polarization, in particular a $(-1)^w$-symmetric nondegenerate pairing. Recall also that $Q(H^{p,q},H^{p',q'})=0$ for $p'\neq w-p$, so $Q(F^pH\hb^{-p},F^qH\hb^{-q})\subset\hb^{-w}\CC[\hb]$. Let $H_{\kk,\bbS}$ be the constant local system on $\bbS$. We have a canonical identification $H_{\kk,\bbS}=\iota^{-1}H_{\kk,\bbS}$. Then $Q_\rB$ induces a $(-1)^w$-$\iota$\nobreakdash-symmetric nondegenerate pairing $\ccQ_\rB$ on $H_{\kk,\bbS}$ in a natural way. The associated pairing $\ccQ:R_FH\otimes_{\CC[\hb]}\nobreak\iota^*R_FH\to(\hb^{-w}\CC[\hb],d,\kk)$ is the pairing generated over $\CC[\hb]\otimes\iota^*\CC[\hb]$ by $Q$. Then $(\cH,\nabla,H_{\kk,\bbS},\ccQ)$ is a pure and polarized \nckHodge structure of weight~$w$.

Indeed, for each $p$, let $\epsilong_p$ be a basis of $H^{p,w-p}$ which is orthonormal for $Q(C\cbbullet,\ov\cbbullet)$, where $C$ is the Weil operator. Recall that $R_FH=\bigoplus_pH^{p,w-p}\hb^{-p}\CC[\hb]$. The family $(\epsilong_p\hb^{-p})_p$ is a $\CC[\hb]$-basis of $R_FH$. We have, by definition and because $\sigma^*\ov\hb=-1/\hb$,
\begin{align*}
\ccC(\epsilong_p\hb^{-p},\sigma^*\ov{\epsilong_q\hb^{-q}})&=i^{-w}Q(\epsilong_p,\ov\epsilong_q)\hb^{-p}\sigma^*\ov\hb^{-q}\\
&=\begin{cases}
0&\text{if }p\neq q,\\
Q(C\epsilong_p,\ov\epsilong_p)&\text{if }p=q,
\end{cases}
\end{align*}
so the family $(\epsilong_p\hb^{-p})_p$ is an orthonormal basis for $\ccC$, hence the polarization property, according to the remark in \S\ref{num:polncHS}.
\end{exemple*}

\begin{exemple*}[Polarization of $\TT_\kk(\ell)$ for $\ell\in\ZZ$, \cf\cite{Mochizuki08b}]
We will use the expression \eqref{eq:Tate} and the previous example. We have a natural isomorphism $\iota^*\TT_\kk(\ell)=\TT_\kk(\ell)$ and a natural nondegenerate symmetric pairing
\[
\ccS:\TT_\kk(\ell)\otimes\TT_\kk(\ell)\to\TT_\kk(2\ell)
\]
induced by the natural product pairing $\kk_\bbS\otimes\kk_\bbS\to\kk_\bbS$. With this in mind, $\ccQ_\rB:(2\pi i)^{\ell}\kk_\bbS\otimes(2\pi i)^{\ell}\kk_\bbS\to\kk_\bbS$ is defined as $(2\pi i)^{-2\ell}$ times the product. For $\ell\in\hZZ$, the polarization will be indicated in Remark \ref{rem:Tatepol} below.
\end{exemple*}

\num{Duality and tensor products}\label{num:dualtens}
Duality and tensor product are well defined both for connections with a pole of order two and for $\kk$-Stokes-filtered local systems, and are compatible via the Riemann-Hilbert correspondence (\cf \cite{Malgrange91}, \cf also \cite{Bibi10}). These functors preserve the category of \nckHodge structures, since they correspond to the corresponding functors for the bundles $(\wt\cH,\wt\nabla)$, and thus transform weights as expected. Similarly, polarization behaves in a natural way. For instance, for $\ell\in\ZZ$, the Tate twist
\[
((\cH,\nabla),(\cL_\kk,\cL_{\kk,\bbullet}),\ccQ_\rB)(\ell)\defin((\cH,\nabla),(\cL_\kk,\cL_{\kk,\bbullet}),\ccQ_\rB)\otimes\TT_\kk(\ell)
\]
consists in replacing $(\cH,\nabla)$ with $(\hb^\ell\cH,\nabla)$, $(\cL_\kk,\cL_{\kk,\bbullet})$ with $(2\pi i)^\ell(\cL_\kk,\cL_{\kk,\bbullet})$ and~$\ccQ_\rB$ with $(2\pi i)^{-2\ell}\ccQ_\rB$.

\num{Comparison with the linear algebra approach of \S\ref{subsec:linalg}}\label{subsec:complinalg}
We now forget about the Stokes structure. Tensoring $(\cH,\nabla,\cL_\kk)$ with $\TT_\kk(\ell)$ ($\ell\in\nobreak\hZZ$) allows one to only treat \nckHodge structures of weight~$0$. A similar reduction can be done for the case of complex \ncHodge structures, as well as for that of polarized \nckHodge structures (by using Remark \ref{rem:Tatepol} below). In the linear algebra approach of \S\ref{subsec:linalg}, the corresponding reduction simply amounts to setting $w=0$ in the objects. We will also assume that $\kk=\RR$, since only this case can occur in \S\ref{subsec:linalg}. 

\begin{lemme*}[{\cf\cite[Th\ptbl2.19]{Hertling01}}]
The functor which associates to $(\cH,\nabla,\cL_\RR)$ pure of weight~$0$ (\resp to a pure complex \ncHodge structure $(\cH,\nabla,\ccC)$ of weight $0$, \resp to a pure polarized TERP structure $(\cH,\nabla,\cL_\RR,\ccP)$ of weight $0$) the object $(\Gamma(\PP^1,\wt\cH),\wt\nabla)$ (\resp $(\Gamma(\PP^1,\wh\cH),\wh\nabla)$, \resp...) is an equivalence with the corresponding linear algebra data $(H,\nabla)$ with a connection \eqref{eq:assconn}.\qed
\end{lemme*}

\num{Rescaling}\label{num:rescaling}
C\ptbl Hertling \cite[Th\ptbl7.20]{Hertling01} has considered the action of $\CC^*$ on the category of connections  $(\cH,\nabla)$ with a pole of order two obtained by rescaling the variable $\hb$. For $x\in\CC^*$, consider the map $\mu_x:U_0\to U_0$ defined by $\mu_x(\hb)=x\hb$. The rescaled connection is $\mu_x^*(\cH,\nabla)$. Since $\mu_x^{-1}(\bbS)=\{\hb\mid |x\hb|=1\}\neq\bbS$ if $|x|\neq1$, we define the pull-back local system $\mu_x^{-1}\cL$ by working with local systems on $U_0\cap U_\infty$ (recall that the inclusion $\bbS\hto U_0\cap U_\infty$ induces an equivalence of categories of local systems on the corresponding spaces). This approach is taken up by C\ptbl Hertling and Ch\ptbl Sevenheck in \cite{H-S06}.

The rescaling acts on the category of objects $(\cH,\nabla,\cL_\RR)$ (by the same procedure as above), on the category of objects $(\cH,\nabla,\ccC)$ since $\iota$ commutes with $\mu_x$, and similarly on the category of objects $(\cH,\nabla,\cL_\kk,\ccQ_\rB)$ as in \S\ref{sec:ncHodge} (without paying attention to the \ncHodge property at the moment). It~also acts on Stokes-filtered local systems $(\cL_\kk,\cL_{\kk,\bbullet})$ in a way compatible, by the Riemann-Hilbert correspondence, to the action on the meromorphic bundles $(\cG,\nabla)$: the subsheaf $(\mu_x^{-1}\cL)_{\kk,\leq c}$ is defined as $\mu_x^{-1}(\cL_{\kk,\leq c/x})$.

If $(\cH,\nabla,\cL_\RR)$ is pure of weight~$w$ (\resp if $(\cH,\nabla,\ccC)$ is pure of weight $0$ and polarized, \resp if $(\cH,\nabla,\cL_\kk,\ccQ_\rB)$ is pure of weight~$w$ and polarized) then, \textit{provided $|x-1|$ is small enough}, the corresponding rescaled object remains pure (\resp pure and polarized) of the same weight: this follows from the rigidity of trivial bundles on $\PP^1$. On the other hand, this may not remain true for all values of the rescaling parameter~$x$.

The subcategory of pure polarized \ncHodge structures which remain so by rescaling by any $x\in\CC^*$ is a global analogue of a nilpotent orbit in the theory of variation of polarized Hodge structures. It has been extensively studied in \cite{H-S06, H-S08b} and \cite{Mochizuki08b}. On the other hand, on this subcategory, the \textit{no ramification} condition for the connection is a consequence of an expected nice behaviour at the limiting values $x=0,\infty$ (this is discussed in \cite[App\ptbl B]{Bibi08}). This would lead to the definition of the category of \textit{rescalable} polarizable pure \ncHodge structures, which would be particular cases of variations of polarizable pure \ncHodge structures on $\CC^*$ with a wild (\resp tame) behaviour at~$x=0$ (\resp $x=\infty$).

\num{Exponential Hodge structures}\label{num:expHS}
In a similar spirit, Kontsevich and Soibelman \cite{K-S10} develop the notion of \textit{exponential mixed Hodge structure}. Given a holonomic $\Clt$-module~$M$, the $\Clt$-module $\CC[t]\langle\partial_t,\partial_t^{-1}\rangle\otimes_{\Clt}M$ is still holonomic and its Laplace transform is equal to the localization $\CC[\tau,\tau^{-1}]\otimes_{\CC[\tau]}\Fou M$. If $M$ is regular holonomic, the corresponding operation on its de~Rham complex, via the Riemann-Hilbert correspondence, is the convolution with  $j_!\QQ_{\Afu\moins\{0\}}$, where $j:\Afu\moins\{0\}\hto\Afu$ denotes the inclusion, as defined in \cite{Katz96}. The convolution with $j_!\QQ_{\Afu\moins\{0\}}$ can be done within the frame of mixed Hodge modules defined in \cite{MSaito87}. Then, given a pure Hodge module on $\Afu$, its convolution with $j_!\QQ_{\Afu\moins\{0\}}$ is mixed, but is declared to be ``exponentially pure'' of the same weight. Theorem \ref{th:FLTQ} below and the remark which follows imply that the Fourier-Laplace transform of an exponential mixed Hodge structure (\ie a mixed Hodge $\cD$-module on $\Afu$ whose de~Rham hypercohomology spaces are zero) is the localization at $\tau=0$ of a rescalable polarized pure \ncHodge structure.

\section{Non-commutative Hodge structures from Deligne-Malgrange lattices}\label{sec:DM}

\num{The basic construction}\label{num:basicconst}
Let $C\subset\CC$ be a finite set, let $\theta_o\in\RR/2\pi\ZZ$ be generic with respect to~$C$ (\cf \S\ref{num:BettiStokesdata}) defining thus a numbering $\{c_1,\dots,c_n\}$ of the set~$C$ in strictly increasing order, and let $\Sigma$ be a \textit{block-lower triangular invertible square matrix} of size $d$ with entries in $\kk\subset\RR$, the blocks being indexed by $C$ ordered by~$\theta_o$. Under some assumptions on $\Sigma$, we will associate to these data and to each integer~$w$ a connection with a pole of order two $(\cH,\nabla)$ with $\kk$-Betti structure $(\cL_\kk,\cL_{\kk,\bbullet})$ and a nondegenerate $(-1)^w$\nobreakdash-$\iota$\nobreakdash-symmetric nondegenerate $\kk$-pairing $\ccQ$ as in \S\ref{num:polnckHodge}, giving rise in particular to a TERP$(-w)$-structure. We will denote these data by $\textup{ncH}(C,\theta_o,\Sigma,w)$.

The matrix $\Sigma$ determines Stokes data $((L_{c,1},L_{c,2}),S,S')$ of type $(C,\theta_o)$ (\cf\eqref{eq:catStokesdata}) by setting $L_1=L_2=\kk^d$, and $L_{c,j}$ ($c\in C$, $j=1,2$) correspond to the blocks of $\Sigma$, which defines a linear morphism $S:L_1\to L_2$, and we define $S'$ as the linear morphism attached to $\Sigma'\defin(-1)^w\cdot{}^t\Sigma$. These Stokes data in turn correspond to a Stokes filtered local system $(\cL_\kk,\cL_{\kk,\bbullet})$. The underlying local system $\cL_\kk$ is completely determined by the $\kk$-vector space $L_1=\kk^d$ together with (the conjugacy class of) its monodromy, whose matrix is $(-1)^w\Sigma^{-1}\cdot{}^t\Sigma$. On the other hand, each diagonal block~$\Sigma_{c_i}$ of~$\Sigma$ gives rise to an invertible matrix $(-1)^w\Sigma_{c_i}^{-1}\cdot{}^t\Sigma_{c_i}$, which represents the monodromy of the meromorphic connection corresponding to $\wh\cG_{c_i}$ in the decomposition \eqref{eq:formalization}.

A nondegenerate $\iota$\nobreakdash-pairing $\ccQ_\rB$ on $(\cL_\kk,\cL_{\kk,\bbullet})$ with values in $\kk$ is determined by a pair of nondegenerate pairings $\ccQ_{12}:L_{1,\kk}\otimes L_{2,\kk}\to\kk$ and $\ccQ_{21}:L_{2,\kk}\otimes L_{1,\kk}\to\kk$ which satisfy $\ccQ_{21}(x_2,x_1)=\ccQ_{12}(S^{-1}x_2,S'x_1)$, and the $(-1)^w$-$\iota$\nobreakdash-symmetry amounts to $\ccQ_{21}(x_2,x_1)=(-1)^w\ccQ_{12}(x_1,x_2)$ (\cf\cite[(3.3)\,\&\,(3.4)]{H-S09}). In the fixed bases of~$L_1$ and~$L_2$, we define\footnote{\label{footnote1}Compared with the definition of $h_{1\ov2}$ and $h_{2\ov1}$ in \cite[\S3c\,\&\,3.e]{H-S09}, which is the case where~$w$ is odd, we changed the sign in order to have a better correspondence with the $\kk$\nobreakdash-structure and the example of \S\ref{num:polnckHodge}.} $\ccQ_{21}(x_2,x_1)={}^tx_2\cdot x_1$, so that $\ccQ_{12}(x_1,x_2)={}^tx_1{}^t\Sigma \cdot\Sigma^{\prime-1}x_2$; the $(-1)^w$-symmetry follows from $\Sigma'=(-1)^w\cdot{}^t\Sigma$. From the Riemann-Hilbert correspondence we finally obtain a nondegenerate $(-1)^w$-$\iota$\nobreakdash-symmetric pairing
\bgroup\numstareq
\begin{equation}\label{eq:QRH}
\ccQ_\rB:((\cG,\nabla),(\cL_\kk,\cL_{\kk,\bbullet}))\otimes\iota^*((\cG,\nabla),(\cL_\kk,\cL_{\kk,\bbullet}))
\to(\CC\lap\hb\rap,d,\kk_{\bbS}). 
\end{equation}
\egroup

We will set (using the notation of \S\ref{subsec:DML}, and stressing upon the fact that the construction of $(\cG,\nabla)$ and $(\cL,\cL_\bbullet)$ above depends on the parity of $w$):
\bgroup\numstarstareq
\begin{equation}\label{eq:ncH}
\textup{ncH}(C,\theta_o,\Sigma,w)=(\DM^{-(w+1)/2}(\cG,\nabla),(\cL_\kk,\cL_{\kk,\bbullet}),\ccQ_\rB).
\end{equation}
\egroup
Let us note that Corollary \ref{cor:DMA} reads as follows:

\begin{corollaire}\label{cor:dual}
Let $(C,\theta_o,\Sigma,w)$ be as above. Assume that $\ker(\Sigma_{c_i}+{}^t\Sigma_{c_i})=0$ for all~$i$. Then $\DM^{>-(w+1)/2}=\DM^{-(w+1)/2}$ and the pairing $\ccQ_\rB$ given by \eqref{eq:QRH} induces a nondegenerate $(-1)^w$-$\iota$\nobreakdash-symmetric pairing, also denoted by $\ccQ_\rB$:
\[
\DM^{-(w+1)/2}(\cG,\nabla)\otimes_{\CC\{\hb\}}\iota^*\DM^{-(w+1)/2}(\cG,\nabla)\to(\hb^{-w}\CC\{\hb\},\rd).
\]
\end{corollaire}

\begin{theoreme}[\cf \cite{H-S09}]\label{th:HS}
Let $(C,\theta_o,\Sigma,w)$ be as in \S\ref{num:basicconst}. We moreover assume the following:
\begin{enumerate}
\item\label{th:HS1}
for each $c\!\in\! C$, the diagonal block $\Sigma_c$ of $\Sigma$ satisfies \hbox{$\ker(\Sigma_c+{}^t\Sigma_c)=0$},
\item\label{th:HS2}
the quadratic form $\Sigma+{}^t\Sigma$ is positive semi-definite.
\end{enumerate}
Then $\textup{ncH}(C,\theta_o,\Sigma,w)$ is a pure and polarized \nckHodge structure of weight~$w$.
\end{theoreme}

\begin{proof}
The condition \ref{th:HS}\eqref{th:HS1} implies that the assumption of Corollary \ref{cor:dual} holds. The case where $w=-1$ is contained in the statement of \cite[Th\ptbl5.9]{H-S09}, where $\ccQ_\rB$ above corresponds to $-\varh_\rB$ there (\cf Footnote \eqref{footnote1}), and where the assumption \ref{th:HS}\eqref{th:HS1} implies $K_c=0$ for each $c\in C$ in \cite[(5.9)($**$)]{H-S09}. The proof of \cite[Th\ptbl5.9]{H-S09} consists in expressing $\textup{ncH}(C,\theta_o,\Sigma,-1)$ as the Fourier-Laplace transform of a variation of complex polarized Hodge structure of type $(0,0)$ on $\CC\moins C$, and to use Theorem \ref{th:FLT} explained in \S\ref{sec:FLT} below.

In order to treat the general case, we remark that (\cf \eqref{eq:ncH}):
\[
\textup{ncH}(C,\theta_o,\Sigma,w-k)=\textup{ncH}(C,\theta_o,\Sigma,w)\otimes\textup{ncH}(C=\{0\},\theta_o=0,\Sigma=1,w=-k),
\]
for each $k\in\ZZ$. According to \S\ref{num:dualtens}, it is thus enough to prove that  the rank-one object $\textup{ncH}(C=\{0\},\theta_o=0,\Sigma=1,w=-k)$ is pure and polarized of weight $-k$ and, by iterating tensor products, it is enough to consider $k=1$.

Let us compute the polarization $\ccQ^{(1/2)}$ for the object $(\cO_{U_0},\rd+(1/2)\rd\hb/\hb,\hb^{-1/2}\kk_\bbS)$, which is seen to correspond to $C=\{0\}$, $\theta_o=0$, $w=-1$. Then $L_1=\kk\cdot e_1$ is the space of sections of $\cL_\kk$ on $(0,\pi)$ and $L_2=\kk\cdot e_2$ on $(-\pi,0)$. Moreover, $\ccQ^{(1/2)}_{21}(e_2,e_1)=1$ and $\ccQ^{(1/2)}_{12}(e_1,e_2)=-1$. The sections $\exp(i\theta/2)e_1\in\cO_{U_0|(0,\pi)}$ and $\exp(i\theta/2)e_2\in\cO_{U_0|(-\pi,0)}$ are the restriction of the $\cO_{U_0}$-basis $v_o=1\in\cO_{U_0}$: indeed, $v_o$ restricted to $U_0\moins\{\hb\in\RR\}$ is written~$\hb^{1/2}e_i$, $i=1,2$. We have $\ccP^{(1/2),\nabla}=i\ccQ^{(1/2)}$ and $\ccC^{(1/2)}$ is the $\sigma$-sesquilinear extension of~$\ccP^{(1/2)}$.

The map $\hb\mto\sigma(\ov\hb)=-1/\hb$ restricts to $\RR/2\pi\ZZ$ as
\begin{align*}
(0,\pi)&\to(0,\pi)&(-\pi,0)&\to(-\pi,0)\\
\theta&\mto-(\theta-\pi)&\theta&\mto-(\theta+\pi)
\end{align*}
We then have
\[
\ccC^{(1/2)}_{|\bbS}(v_o,\sigma^*\ov v_o)=
\begin{cases}
e^{i\theta/2}\cdot e^{-i(\theta-\pi)/2}\cdot i\cdot \ccQ^{(1/2)}_{12}(e_1,e_2)=1&\text{for }\theta\in(0,\pi),\\
e^{i\theta/2}\cdot e^{-i(\theta+\pi)/2}\cdot i\cdot \ccQ^{(1/2)}_{21}(e_2,e_1)=1&\text{for }\theta\in(-\pi,0).
\end{cases}
\]
This shows that $(v_o,\sigma^*\ov v_o)$ is an orthonormal basis for $\ccC^{(1/2)}$, hence (\cf the remark in \S\ref{num:polncHS}) the corresponding $(\cO_{U_0},\rd+\frac12\rd\hb/\hb,\ccC^{(1/2)})$ is pure of weight~$0$ and polarized.
\end{proof}

\begin{remarque}\label{rem:Tatepol}
By rotating the $\QQ$-structure, the last argument can be used similarly to show that the Tate pure \nckHodge structure $\TT_\kk(\ell)$ of weight $-2\ell$ is polarized, and make explicit its polarization.
\end{remarque}

\section{Non-commutative $\kk$-Hodge structures by Fourier-Laplace~transformation}\label{sec:FLT}

\num{}
Let $C\subset\Afu$ be a finite set of points on the complex affine line with coordinate~$t$. Let $(\ccV_\kk,F^\cbbullet V, \nabla, Q_\rB)$ be a variation of polarized Hodge structure of weight $w\in\ZZ$ on $X\defin \Afu\moins C$. Namely,
\begin{itemize}
\item
$(V,\nabla)$ is a holomorphic vector bundle with connection on~$X$,
\item
$F^\cbbullet V$ is a finite decreasing filtration of $V$ by holomorphic sub-bundles satisfying the Griffiths transversality property: $\nabla F^pV\subset F^{p-1}V\otimes_{\cO_X}\Omega^1_X$,
\item
$\ccV_\kk$ is a $\kk$-local system on $X$ with $\ccV_\kk\otimes_\kk\CC=V^\nabla$,
\item
$Q_\rB:\ccV_\kk\otimes_\kk\ccV_\kk\to\kk$ is a nondegenerate $(-1)^w$-symmetric pairing,
\end{itemize}
all these data being such that the restriction at each $x\in X$ is a polarized Hodge structure of weight~$w$ (\cf \eg\cite{P-S08} or the example in \S\ref{num:polnckHodge}). We denote by $Q$ the nondegenerate pairing $(V,\nabla)\otimes(V,\nabla)\to\cO_X$ that we get from $Q_\rB$ through the canonical isomorphism $\cO_X\otimes_\kk\ccV_\kk=V$. The associated nondegenerate sesquilinear pairing is denoted by $k:(V,\nabla)\otimes_\CC\ov{(V,\nabla)}\to\nobreak\cC_X^\infty$, which is obtained from $k_\rB:\ccV\otimes_\CC\ov\ccV\to\CC$ similarly. It is $(-1)^w$-Hermitian and $i^{-w}k$ induces a flat Hermitian pairing on the $C^\infty$\nobreakdash-bundle $(\cC_X^\infty\otimes_{\cO_X}V,\nabla+\ov\partial)$. We can regard $(V,\nabla, F^\cbbullet V,i^{-w}k)$ as a variation of polarized complex Hodge structure, pure of weight~$0$.

\num{The middle (or minimal) extension}\label{num:minimalext}
This procedure transforms the previous data, defined on $\Afu\moins C$, into similar data defined on $\Afu$. Let $j:X\hto\Afu$ denote the inclusion.

\textit{Betti side}: The pairing $Q_\rB$ (\resp $k_B$) extends in a unique way as a nondegenerate $(-1)^w$-symmetric (\resp $(-1)^w$-Hermitian) pairing $j_*Q_\rB:j_*\ccV_\kk\otimes_\kk\nobreak j_*\ccV_\kk\to\nobreak\kk_{\Afu}$ (\resp $j_*k_B:j_*\ccV_\CC\otimes_\CC\nobreak j_*\ccV_\CC\to\nobreak\CC_{\Afu}$).

\textit{De~Rham side}:  The bundle $(V,\nabla)$ can be extended in a unique way as a free $\cO_{\PP^1}(*C\cup\{\infty\rap$-module with a connection $\nabla$ having a regular singularity at $C\cup\nobreak\{\infty\}$ (Deligne meromorphic extension). Taking global sections on~$\PP^1$ produces a left module~$\wt M$ on the Weyl algebra $\Clt$. The minimal extension (along $C$) of $\wt M$ is the unique submodule $M$ of $\wt M$ which coincides with $\wt M$ after tensoring both by $\CC(t)$, and which has no quotient submodule supported in $C$. The pairing $k$ extends first (due to the regularity of the connection) as a pairing $\wt k:\wt M\otimes_\CC\ov{\wt M}\to\cS'(\Afu\moins C)$, where $\cS'(\Afu)$ denotes the Schwartz space of temperate distributions on $\Afu=\RR^2$, and $\cS'(\Afu\moins C)\defin\CC\big[t,\prod_{c\in C}(t-c)^{-1}\big]\otimes_{\CC[t]}\cS'(\Afu)$. Then one shows that, when restricted to $M\otimes_\CC\ov M$, $\wt k$ takes values in $\cS'(\Afu)$, and we denote it by~$k$ (\cf\cite[\S1.a]{H-S09}, where no distinction is made between $h$ and~$k$ since we only deal there with variations of Hodge structures of type $(p,p)$ for some $p$).

\textit{Hodge side}: The Hodge filtration $F^\cbbullet V$ extends, according to a procedure due to M\ptbl Saito \cite[\S3.2]{MSaito86}, to a good filtration $F^\cbbullet M$ of $M$ as a $\Clt$-module (\cf \cite[\S3.d]{Bibi05}).

\num{Fourier-Laplace transformation}
The Fourier transformation $\rF_t:\cS'(\Afu_t)\to\cS'(\Afu_{\hb'})$ with kernel $\exp(\ov{t\hb'}-t\hb')\itwopi\,dt\wedge d\ov t$ is an isomorphism from the Schwartz space $\cS'(\Afu_t)$ considered as a $\Clt\otimes_\CC\CC[\ov t]\langle\partial_{\ov t}\rangle$-module, to $\cS'(\Afu_{\hb'})$ considered as a $\CC[\hb']\langle\partial_{\hb'}\rangle\otimes_\CC\CC[\ov\hb']\langle\partial_{\ov\hb'}\rangle$-module.

Composing $k$ with $\rF_t$ defines a sesquilinear pairing $\Fou k:\Fou M\otimes_\CC\iota^+\ov{\Fou M}\to\cS'(\Afu_{\hb'})$, where $\Fou M$ is the Laplace transform of $M$ as in \S\ref{num:Laplace}, and where~$\iota^+$ denotes the pull-back by $\iota$ in the sense of $\CC[\hb']\langle\partial_{\hb'}\rangle$-modules (or $\CC[\hb']$-modules with connection). See \cite[\S1.a]{Bibi05} for the need of $\iota^+$.

Restricting to $\CC^*$ produces a sesquilinear pairing $\Fou k:(\cG,\nabla)\otimes\nobreak\iota^*\ov{(\cG,\nabla)}\to(\cC^\infty_{\CC^*},\rd)$, whose horizontal part restricted to $\bbS$ defines a pairing $\ccC^\nabla_\bbS:\cL\otimes\nobreak\iota^{-1}\ov\cL\to\CC_\bbS$ (we use here the notation of \S\ref{num:Laplace}). We then define~$\ccC$ as in \S\ref{num:polncHS}.

The pairing $\Fou k$ restricts to horizontal sections of $(\cG,\nabla)$ to produce a Betti $\iota$\nobreakdash-sesquilinear pairing $(\Fou k)_\rB$ on $\cL$. It is defined only over $\CC^*$. On the other hand, in a way similar to Proposition \ref{prop:whQ}, there is a topological Fourier-Laplace transform $\wh{j_*k_\rB}$, which is compatible with the Stokes filtration. The comparison between both is given by:

\begin{lemme}[{\cf \cite[Prop\ptbl1.18]{Bibi05} \& \cite[Appendix]{H-S09}}]\label{lem:compkk}
Over $\CC^*$ we have $(\Fou k)_\rB=\itwopi\wh{j_*k_\rB}$.\qed
\end{lemme}

Lastly, we denote by $\cH\subset\cG$ the Brieskorn lattice of the good filtration $F^\cbbullet M$ (\cf\S\ref{num:minextBr}), and we recall that the connection has a pole of order two at most on $\cH$.

\begin{theoreme}[{\cite[Cor\ptbl3.15]{Bibi05}}]\label{th:FLT}
Let $(\ccV_\kk,F^\cbbullet V, \nabla, Q_\rB)$ be a variation of polarized Hodge structure of weight $w\in\ZZ$ on $X\defin \Afu\moins C$. Then $(\cH,\nabla,i^{-w}\ccC)$ defined as above is a pure polarized complex \ncHodge structure of weight~$0$.\qed
\end{theoreme}

We now make more precise Theorem \ref{th:FLT}, which only produces a polarized complex \ncHodge structure, in order to get a polarized \nckHodge structure. Recall that the pairing $\wh{j_*Q_\rB}$ has been considered in Proposition \ref{prop:whQ}. We notice that the topological Laplace transform $\wh{j_*k_\rB}$ is nothing but the $\iota$\nobreakdash-sesquilinear pairing associated with the $\iota$\nobreakdash-pairing $\wh{j_*Q_\rB}$ on $\cL_\kk$.

\begin{theoreme}\label{th:FLTQ}
Let $(\ccV_\kk,F^\cbbullet V, \nabla, Q_\rB)$ be a variation of polarized $\kk$-Hodge structure of weight $w\in\ZZ$ on $X\defin \Afu\moins C$. Then $((\cH,\nabla),(\cL_\kk,\cL_{\kk,\bbullet}),-\wh{j_*Q_\rB})$ is a pure polarized \nckHodge structure of weight~$w+1$.
\end{theoreme}

\begin{remarque*}
One can show (see \cite[Rem\ptbl2.5]{Bibi05}) that the pure polarized \nckHodge structures that one gets by Fourier-Laplace transformation are rescalable, in the sense given in \S\ref{num:rescaling}.
\end{remarque*}

\begin{remarque*}
In order to go from Theorem \ref{th:FLT} to Theorem \ref{th:FLTQ}, we need some more work on \textit{bilinear} pairings (while Lemma \ref{lem:compkk} and Theorem \ref{th:FLT} only use \textit{sesquilinear} pairings at the topological or analytical level). The pair $(\ccQ,\ccQ_\rB)$ of \eqref{eq:polnckHodge*} that we wish to use consists of $\ccQ_\rB=-\wh{j_*Q_\rB}$ and of a constant multiple of the Laplace transform of the algebraic duality isomorphism of the $\Clt$-module $M$ associated with $(V,\nabla)$ as in \S\ref{num:minimalext}. We first check that this algebraic duality is compatible with filtrations (Lemma \ref{lem:AppA}), hence satisfies the holomorphic part of \eqref{eq:polnckHodge*}, by using properties of Hodge $\cD$-modules (Corollary \ref{cor:FouQ}). We then show that the pairing corresponding to $-\wh{j_*Q_\rB}$ by the Riemann-Hilbert correspondence essentially coincides with this algebraic pairing (Lemma \ref{lem:AppB}). Moreover, notice that Lemma \ref{lem:compkk} only holds over $\CC^*$, while Lemma \ref{lem:AppB} holds including at $\hb=0$. In particular, Lemma \ref{lem:AppB} and Lemma \ref{lem:compkk} are not of the same nature and are independent one from the other.
\end{remarque*}

\begin{proof}[\proofname\ of Theorem \ref{th:FLTQ}]
Let $(\ccV_\kk,F^\cbbullet V, \nabla, Q_\rB)$ be a variation of polarized Hodge structure of weight~$w$ on $X$. According to \cite[Th\ptbl2]{MSaito86}, it corresponds to a polarized Hodge module which is pure of weight $w+1$. It extends in a unique way as a polarized Hodge module on $\Afu$, with underlying filtered $\Clt$-module $(M,F^\cbbullet M)$ as considered in \S\ref{num:minimalext} above. The polarization induces a $(-1)^{w+1}$-symmetric isomorphism $Q:(M,F^\cbbullet M)\isom\bD(M,F^\cbbullet M)(-(w+\nobreak1))$ (\cf \loccit for the duality of filtered $\cD$-modules, and Lemma 5.2.12 there for the polarization, which is denoted by $S'$ and corresponds to $(2\pi i)^{-w}Q$), and where, for a filtered module $(N,F^\cbbullet N)$, we set $(N,F^\cbbullet N)(k)=(N,F^{\cbbullet+k} N)$. It is easy to check, from the very definition of the Brieskorn lattice of a good filtration (\cf \cite[\S1.d]{Bibi05}), that the Brieskorn lattice of $(N,F^\cbbullet N)(k)$ is equal to $\hb^k\cH$, where~$\cH$ is the Brieskorn lattice of $(N,F^\cbbullet N)$.

On the other hand, according to the relation between duality and Laplace transforms of $\Clt$-modules (\cf \cite[Lem\ptbl3.6,\,p\ptbl86]{Malgrange91}, see also \cite[\S V.2.b]{Bibi00}), we have $\Fou(D M)=\iota^+D\Fou M$ (where $i^+$ denotes the pull-back of $\CC[\hb']\langle\partial_{\hb'}\rangle$-modules and $D$ is the duality of holonomic $\Clt$-modules), and thus, according to \cite[\S2.7]{MSaito89} (\cf also \cite[Lem\ptbl3.8]{Bibi96bb}), the localized Laplace transform of $D M$ is identified with $\iota^*\cG^\vee$ (\cf Lemma \ref{lem:dualDM} for the notation). 

Given a holonomic $\Clt$-module with a good filtration $(M,F^\cbbullet M)$, we say that $\bD(M,F^\cbbullet M)$ is strict (\cf \cite{MSaito86}) if the dual complex $\bD R_FM$ (as $R_F\Clt$-modules, \cf \S\ref{app:BL}) has cohomology in degree one only, without $\CC[\hb]$-torsion. This cohomology can then be written in a unique way as $R_FDM$ for some good filtration $F^\bbullet DM$.

For a polarizable Hodge module, $\bD(M,F^\cbbullet M)$ is strict (\cf \cite{MSaito86}), and the polarization gives a filtered isomorphism between $M$ and $DM$ as above.

The proof of the following lemma will be sketched in Appendix \ref{app:proofFLTQ}.

\begin{lemme}\label{lem:AppA}
Assume that $\bD(M,F^\cbbullet M)$ is strict. Then, through the previous identification of the localized Laplace transform of $D M$ with $\iota^*\cG^\vee$, the Brieskorn lattice of $(DM,F^\cbbullet DM)$ is identified with $\iota^*\cH^\vee$.
\end{lemme}

\begin{corollaire}\label{cor:FouQ}
The morphism $\Fou Q$ induces an isomorphism $(\cG,\nabla)\isom\iota^*(\cG,\nabla)^\vee$ which sends $\cH$ onto $\hb^{-(w+1)}\iota^*\cH^\vee$, and thus induces a nondegenerate $(-1)^{w+1}$-$\iota$\nobreakdash-symmetric pairing $\ccQ$ on $\cH$.\qed
\end{corollaire}

Let us now denote by $\wh{j_*Q}$ the nondegenerate pairing $(\cG,\nabla)\isom\iota^*(\cG,\nabla)^\vee$ induced by $\wh{j_*Q_\rB}$ through the inverse RH correspondence.

The proof of the following lemma will be sketched in Appendix \ref{app:B}.
\begin{lemme}\label{lem:AppB}
The pairings $\Fou Q$ and $\wh{j_*Q}$ coincide up to a multiplicative constant.
\end{lemme}

Let us set $\ccQ_\rB=-\wh{j_*Q_\rB}$. Then, according to Lemma \ref{lem:AppB} and Corollary \ref{cor:FouQ}, $(\ccQ,\ccQ_\rB)$ satisfies \eqref{eq:polnckHodge*} with $w+1$ instead of~$w$. For the polarizability property, let us set $\ccC_\bbS^\nabla\defin i^{-(w+1)}\ccQ_\bbS^\nabla$ made sesquilinear. We thus have $\ccC_\bbS^\nabla=-i^{-(w+1)}\wh{j_*k_B}=2\pi\cdot i^{-w}(\Fou k)_\rB$, according to Lemma \ref{lem:compkk}. Hence Theorem \ref{th:FLT} gives the polarizability.
\end{proof}

\section{The {\let\kk\QQ\protect\nckHodge} structure attached to a tame function}\label{sec:tamefunct}

\num{}
Let $X$ be a complex smooth quasi-projective variety and let $f:X\to\Afu$ be a regular function on it, that we regard as a morphism to the affine line~$\Afu$ with coordinate~$t$. For each $k\in\ZZ$, the  perverse cohomology sheaf $\pcH^k(\bR f_*\QQ_X)$ underlies a mixed Hodge module (\cf\cite{MSaito87}). The fibre at $\hb=\nobreak1$ of its topological Laplace transform (\cf \cite[Chap\ptbl VI,\,\S2]{Malgrange91}) is the $k$-th \textit{exponential cohomology space of $X$ with respect to $f$} (or simply of $(X,f)$).

The \textit{exponential periods} attached to $f$ are the integrals $\int_\gamma e^{-f}\omega$, where~$\omega$ is an algebraic differential form of degree $k$ on $X$ and~$\gamma$ is a $k$-cycle in the Borel-Moore homology of $X$, such that $\reel f >\epsilon>0$ on the support of~$\gamma$ and away from a compact set in $X$.

Formulas like $\int_\RR e^{-x^2}dx=\sqrt\pi$ suggest (\cf \cite[p\ptbl118]{Deligne8406}) to produce a ``Hodge filtration'' with rational or real indices (here $1/2$) on the exponential cohomology of $(X,f)$. This is developed in \loccit for particular examples.

On the other hand, when $f$ is proper, $\pcH^k(\bR f_*\QQ_X)\star j_!\QQ_{\Afu\moins\{0\}}$ is \textit{exponentially pure} in the sense of Kontsevich-Soibelman \cite{K-S10} (\cf\S\ref{num:expHS}), but so is also the case when all the graded object, except one, with respect to the weight filtration are constant, hence killed by the convolution operation. The cohomologically tame case considered below enters this frame.

The non-commutative Hodge structure approach that we explain below consists, when $\pcH^k(\bR f_*\QQ_X)\star j_!\QQ_{\Afu\moins\{0\}}$ is \textit{exponentially pure}, in considering the $k$-th exponential cohomology space of $(X,f)$ as $\Gamma(\PP^1,\wh\cH)$, where $\wh\cH$ is defined in \S\ref{num:polncHS}. In particular, the non-commutative Hodge structure is defined on~$\cH$. The relation with the construction of Deligne quoted above is explained in \cite[\S6]{Bibi08}.

\num{}
For the sake of simplicity, we will only consider the case of a cohomologically tame function $f:U\to\Afu$ on a smooth affine complex manifold $U$, for which there is only one non-zero exponential cohomology space. We will constantly refer to \cite{Bibi96b,Bibi96bb} and \cite{Bibi05}.

Recall (\cf\loccit) that cohomological tameness implies that there exists a diagram
\[
\xymatrix{
U\ar@{^{ (}->}[r]^-\kappa\ar[dr]_(.35)f&X\ar[d]^F\\
&\Afu
}
\]
where $X$ is quasi-projective and $F$ is projective, such that the cone of natural morphism $\kappa_!\QQ_U\to\bR\kappa_*\QQ_U$ has no vanishing cycle with respect to $F-c$ for any $c\in\CC$ (\cf also \cite[Th\ptbl14.13.3]{Katz90}).

We will use the perverse shift convention by setting $\pQQ_U=\QQ[\dim U]$. By Poincaré-Verdier duality, we have a natural pairing
\[
Q_\rB:\bR f_!\pQQ_U\otimes_\QQ\bR f_*\pQQ_U\to\QQ_{\Afu}[2].
\]
Considering the $\QQ$-perverse sheaf $\cF=\pcH^0(\bR f_*\pQQ_U)$, we therefore get a morphism $\DD\cF\to\cF$, whose kernel and cokernel (in the perverse sense) are constant sheaves up to a shift. Let $(\cL_\QQ,\cL_{\QQ,\bbullet})$ be the Stokes-filtered local system on $\bbS$ deduced from the topological Laplace transform of $\cF$ (\cf \S\ref{num:LaplaceBetti}). According to Lemma \ref{prop:whQ}, it comes equipped with a nondegenerate pairing
\let\kk\QQ
\[
\ccQ_B\defin-\wh{j_*Q_\rB}:(\cL_\kk,\cL_{\kk,\bbullet})\otimes\iota^{-1}(\cL_\kk,\cL_{\kk,\bbullet})\to\kk_\bbS.
\]

On the other hand, let $G_0$ denote the Brieskorn lattice of $f$. By definition,
\[
G_0=\Omega^{\dim U}(U)[\hb]\big/(\hb\rd-\rd f\wedge)\Omega^{\dim U-1}(U)[\hb],
\]
and set $G=\CC[\hb,\hbm]\otimes_{\CC[\hb]}G_0$, with the action of $\nabla_{\partial_\hb}$ induced by $\partial_\hb+\nobreak f/\hb^2=e^{f/\hb}\circ\partial_\hb\circ e^{-f/\hb}$ on $\Omega^{\dim U}(U)[\hb]$. We also set $G_k=\hb^{-k}G_0$. For $\ell\in\ZZ$ we set $\epsilon(\ell)=(-1)^{\ell(\ell-1)/2}$.

\begin{theoreme}
The data $((G_{\dim U},\nabla),(\cL_\QQ,\cL_{\QQ,\bbullet}),\epsilon(\dim U-1)\ccQ_B)$ is a polarized \nckHodge structure which is pure of weight $\dim U$.
\end{theoreme}

\begin{proof}[Sketch of proof]
We refer to \cite[Proof of Th\ptbl4.10]{Bibi05}. We first replace the perverse sheaf~$\cF$ defined above with $\cF_{!*}\defin\pcH^0(\bR F_*\kappa_{!*}\pQQ_U)$, which generically is the local system of intersection cohomology of the \hbox{fibres} of~$F$, and we have a corresponding Poincaré-Verdier duality pairing $Q_{\rB,!*}$, whose topological Laplace transform $-\wh{j_*Q_{\rB,!*}}$ coincides with $\ccQ_\rB$. By applying M\ptbl Saito's results on polarizable Hodge $\cD$-modules, together with Theorem \ref{th:FLTQ}, we find that $((G^\rH_0,\nabla),(\cL_\QQ,\cL_{\QQ,\bbullet}),\epsilon(\dim U-\nobreak1)\ccQ_B)$ is a pure polarized \nckHodge structure of weight $\dim U$, where $G_0^\rH$ is the Brieskorn lattice of the Hodge filtration of the Hodge module corresponding to $\cF_{!*}$. By \cite[Lem\ptbl4.7]{Bibi05}, taking also into account the shift between the standard filtration and M\ptbl Saito's Hodge filtration, we have $G_0^\rH=G_{\dim U}$.
\end{proof}

\begin{corollaire}
The data
\[
((G_0,\nabla),(\cL_\QQ,\cL_{\QQ,\bbullet}),\epsilon(\dim U-1)\ccQ_\rB)
\]
is a pure polarized \nckHodge structure of weight $-\dim U$, and the corresponding~$\ccP$  can be written as $i^{-\dim U}\epsilon(\dim U)\wh{j_*Q_\rB}$. \qed
\end{corollaire}

\section{Numerical invariants of \ncHodge structures}

\num{Spectrum at $\hb=\infty$}
To any germ $(\cH,\nabla)$ consisting of a free $\CC\{\hb\}$-module of finite rank equipped with a meromorphic connection (with pole of arbitrary order at $\hb=0$) for which the eigenvalues of the monodromy have absolute value equal to one, is attached a numerical invariant called its spectrum at infinity (\cf\cite[\S III.2.b]{Bibi00} or \cite[\S1.a]{Bibi08}), and encoded as a polynomial $\Sp^\infty_\cH(T)=\prod_\gamma(T-\gamma)^{\nu_\gamma}$, where $\gamma$ varies in $\RR$ (and more precisely, $e^{2\pi i\gamma}$ is an eigenvalue of the monodromy), and $\nu_\gamma\in\NN$. Its behaviour by duality is described in \cite[Prop\ptbl III.2.7]{Bibi00}, and the behaviour with respect to tensor product is better described in terms of the divisor of the polynomial $\Sp^\infty_\cH(T)$, an element of the ring $\ZZ[\RR]$ which is sometimes written as $\sum \nu_\gamma u^\gamma$. Then this divisor behaves in a multiplicative way with respect to tensor product, provided that both terms of the tensor product are essentially self-dual (\cf\cite[Ex\ptbl III.2.9]{Bibi00}).

\num{Spectrum at $\hb=0$}
Assume now that $(\cH,\nabla)$ is of \nrexp, so has a formal decomposition \eqref{eq:formalizationH}. Assume also that the eigenvalues of the monodromy of each $\cH_i$ have absolute value equal to one. Then each $(\cH_i,\nabla_i)$ has a spectrum at the origin, defined with the help of the $V$-filtration (\cf\S\ref{subsec:Brieskornregular}) at $\hb=0$. We use the convention of \cite[Def\ptbl1.7]{Bibi08}. The product of the spectral polynomials $\Sp_{\cH_i}^0(T)$ for all~$i$ is denoted $\Sp_{\cH}^0(T)$. It has properties similar to that of $\Sp^\infty_\cH(T)$ with respect to various operations (\cf\cite[\S III.1.c]{Bibi00} and \cite[1.b]{Bibi08}).

\num{The ``new supersymmetric index'' (\cf \cite{Hertling01})}
Let $\cT=(\cH,\nabla,\ccC)$ be a pure complex \ncHodge structure of weight $0$ (\cf\S\ref{num:polncHS}) and let $\cQ$ be the ``new supersymmetric index'' associated to it through the correspondence of \S\ref{subsec:complinalg}. Its characteristic polynomial will be denoted by $\Susy_\cT(T)$. If $(\cH,\nabla,\ccC)$ is polarized, then $\cQ$ is self-adjoint with respect to the corresponding positive definite Hermitian
form $h$, hence is semi-simple with real eigenvalues, so the roots of $\Susy_\cT(T)$ are real. If $((\cH,\nabla),(\cL_\kk,\cL_{\kk,\bbullet}),\ccQ_\rB)$ is a polarized \ncHodge structure which is pure of weight~$w$, then the ``new supersymmetric index'' of the corresponding pure complex \ncHodge structure of weight $0$ is purely imaginary, hence its (real) eigenvalues are symmetric with respect to the origin.

\begin{exemple*}[\cf {\cite[Lemma 5.4]{Bibi08}}]
For a polarized Hodge structure of weight~$w$ and Hodge numbers $h^{p,w-p}$, we have $\Susy_\cT(T)=\prod_p(T-\nobreak p+\nobreak w/2)^{h^{p,w-p}}$. On the other hand, $\Sp_\cH^\infty(T)=\Sp_\cH^0(T)=\prod_p(T-p)^{h^{p,w-p}}$. 
\end{exemple*}

\nums{Rescaling}
This relationship between $\Sp^0,\Sp^\infty$ and $\Susy$ can be generalized by considering the action of the rescaling (\cf\S\ref{num:rescaling}), when the \ncHodge structure is stable by rescaling, and has a good limiting behaviour, as in the case of a Fourier-Laplace transform of a variation of Hodge structure. More precisely, one gets:
\begin{theoreme}[\cf{\cite[Th\ptbl7.1]{Bibi08}}]
Let $((\cH,\nabla),(\cL_\kk,\cL_{\kk,\bbullet}),\ccQ_\rB)$ be the \nckHodge structure obtained by Fourier-Laplace transformation from a variation of polarized Hodge structure of weight~$w$ on $\Afu\moins C$ (\cf\S\ref{sec:FLT}) and let $\cT$ be the associated polarized complex \ncHodge structure of weight~$0$. Then
\begin{align*}
\Sp^0_\cH(T)&=\lim_{x\to0}\Susy_{\mu_x^*\cT}(T-w/2),\\
\Sp^\infty_\cH(T)&=\lim_{x\to\infty}\Susy_{\mu_x^*\cT}(T-w/2).
\end{align*}
\end{theoreme}

\num{Limit theorems}
The previous theorem is proved by showing that the rescaled objects $\mu_x^*((\cH,\nabla),(\cL_\kk,\cL_{\kk,\bbullet}),\ccQ_\rB)$ form a variation of \nckHodge structures parametrized by $x\in\CC^*$, and $\mu_x^*\cT$ extends as a pure polarized wild twistor $\cD$\nobreakdash-module (\cf\cite{Bibi06b}) on the projective completion $\PP^1\supset\CC^*$ (wildness only occurs at $x=0$). One proves limit theorems in such a setting (\cf \cite{Bibi08}), showing that the limiting twistor structure, when $x\to\infty$, is a mixed Hodge structure polarized by a nilpotent endomorphism, giving the first equality, and, when $x\to0$, this limiting structure decomposes as the direct sum of exponentially twisted mixed Hodge structures polarized by a nilpotent endomorphism, giving the second equality.

A more general approach to these limit theorems, more in the spirit of Schmid's nilpotent orbit theorem, but in both tame and wild cases, and with many parameters, has been obtained by T\ptbl Mochizuki in \cite{Mochizuki08b}, proving thereby a conjecture of C\ptbl Hertling and Ch\ptbl Sevenheck \cite[Conj\ptbl9.2]{H-S06} on nilpotent orbits of pure polarized TERP structures.

\setcounter{section}{1}
\setcounter{equation}{0}
\renewcommand{\thesection}{\Alph{section}}
\renewcommand{\theparagraph}{\Alph{paragraph}}

\section*{Appendix}
\paragraph{Proof of Lemma \ref{lem:AppA}}\label{app:proofFLTQ}

We first make more precise the notion of Brieskorn lattice of a filtered $\Clt$-module, in order to manipulate it more easily. We mainly refer to \cite[\S2.d]{Bibi05}.

\num{The Brieskorn lattice}\label{app:BL}
Let $(M,F_\bbullet M)$ be a holonomic $\Clt$-module equipped with a good filtration (here we use the increasing version of a filtration, in order to be compatible with \cite[\S2.d]{Bibi05}; recall the standard convention $F_p=F^{-p}$ relating increasing and decreasing filtrations). The Rees module $R_FM=\bigoplus_\ell F_\ell M\hb^\ell$ is a module of finite type over the Rees ring $R_F\Clt$, that we identify with the ring $\Clthb$ by setting $\partiall_t=\hb\partial_t$. The Laplace transform $\Fou(R_FM)$ is the $\CC[\hb]$-module $R_FM$ equipped with the structure of a $\Cltauhb$-module, where~$\tau$ acts as $\partiall_t$ and $\partiall_\tau$ as $-t$. It is also of finite type. Moreover, $R_FM$ is also equipped with an action of $\hb^2\partial_\hb$ (\ie is a $\Cltdhb$-module), which is the natural one on $R_FM$, defined as $\hb^2\partial_\hb(m_\ell\otimes\hb^\ell)=\ell m_\ell\otimes\hb^{\ell+1}$. The action of $\hb^2\partial_\hb$ on $\Fou(R_FM)$ (\ie its $\Cltaudhb$-structure) is twisted  as $\hb^2\partial_\hb\Fou(m_\ell\otimes\hb^\ell)=(\partial_tt+\nobreak\ell)m_\ell\otimes\nobreak\hb^{\ell+1}$ (\cf \cite[Rem\ptbl2.2]{Bibi05}). The Brieskorn lattice $G_0^{(F_\bbullet)}$ of $(M,F_\bbullet M)$ is the restriction to $\tau=1$ of $\Fou(R_FM)$, with the induced $\CC[\hb]\langle\hb^2\partial_\hb\rangle$ structure.

Moreover, let $\epsilon:\Cltauhb\to\Cltauhb$ denote the involution $\tau\mto-\tau$, $\partiall_\tau\mto-\partiall_\tau$. The restriction at $\tau=1$ of $\epsilon^*\Fou(R_FM)$ is equal to the restriction at $\tau=-1$ of $\Fou(R_FM)$, and the formulas given in \cite[Lem\ptbl2.1\,\&\,Rem\ptbl2.2]{Bibi05} identify it with $\iota^*G_0^{(F_\bbullet)}$.

\num{Duality and Laplace transformation}
Exactly as in the case of $\Clt$-modules (\cf \cite[Lem\ptbl V.3.6]{Malgrange91}, see also \cite[\S V.2.b]{Bibi00}), the relation \hbox{between} duality and Laplace transformation of $R_F\Clt$-modules is given by $\bD\Fou(R_FM)=\epsilon^*\Fou(\bD R_FM)$. Since $\bD(M,F_\bbullet M)$ is strict, \ie $\bD R_FM$ has cohomology in degree one only, without $\CC[\hb]$-torsion, then $\Fou(\bD R_FM)$ is also strict, hence so is $\bD\Fou(R_FM)$. Moreover, the unique cohomology of $\Fou(\bD R_FM)$ is $\Fou(R_FDM)$, and we denote by $D\Fou(R_FM)$ the unique cohomology of $\bD\Fou(R_FM)$, so $D\Fou(R_FM)=\epsilon^*\Fou(R_FDM)$. Let us now localize with respect to $\tau$ as in \cite[Lem\ptbl2.1]{Bibi05}, \ie apply $\CC[\tau,\tau^{-1},\hb]\otimes_{\CC[\tau,\hb]}$ to both terms, which then become $\CC[\tau,\tau^{-1},\hb]$-free of finite type. After \loccit, the right-hand term is identified with $\CC[\tau,\tau^{-1}]\otimes\iota^*\cH'$, where $\cH'$ denotes the Brieskorn lattice of $(DM,F_\bbullet DM)$. On the other hand, arguing as in \cite[\S2.7]{MSaito89} (\cf also \cite[Lem\ptbl3.8]{Bibi96bb}), the left-hand term is identified with the dual module $(\CC[\tau,\tau^{-1},\hb]\otimes_{\CC[\tau,\hb]}R_FM)^\vee$ with its natural connection. Restricting to $\tau=1$ gives $\cH^\vee$. Therefore, $\iota^*\cH'=\cH^\vee$, as wanted.\qed

\paragraph{Sketch of proof of Lemma \ref{lem:AppB}}\label{app:B}
This kind of comparison result goes back at least to the notion of higher residue pairings, due to K\ptbl Saito, in the theory of singularities of complex hypersurfaces. One finds in \cite[\S2.7]{MSaito89} a similar result, proved by using a universal unfolding of a holomorphic function having an isolated singularity. Another geometric approach, in the present setting, has been proposed by C\ptbl Hertling (unpublished notes) following the geometric construction, due to F\ptbl Pham, of the intersection form on Lefschetz thimbles (\cf \cite[Th\ptbl10.28]{Hertling00} for the analogous result in Singularity theory). We will sketch a sheaf-theoretic proof, by following the Riemann-Hilbert correspondence all along the Fourier-Laplace transformation.

\num{}
The Laplace transform $\Fou M$ of $M$ can be obtained by an ``integral formula'' $\Fou M=\wh p_+(p^+M\otimes E^{-t\hb'})$, where $p,\wh p$ denote the projections $\Afu_t\times\nobreak\Afu_{\hb'}\to\Afu_t,\,\Afu_{\hb'}$, and $E^{-t\hb'}=(\CC[t,\hb'],d-d(t\hb'))$.

\num{}
The duality isomorphism $\Fou(DM)\isom\iota^+D\,\Fou M$ mentioned in the proof of Theorem \ref{th:FLTQ} can be obtained by applying standard isomorphisms of commutation between the duality functor and the proper direct image functor one the one hand, the smooth inverse image functor on the other hand, in the realm of $\cD$-module theory (\cf \eg \cite[VII.9.6 \& VII.9.13]{Borel87b} or \cite[Th\ptbl4.33 \& Th\ptbl4.12]{Kashiwara03}). In order to do so, it is convenient to extend the previous setting to $\PP^1_t\times\PP^1_{\hb'}$ in order to work with a proper map $\wh p$. In such a way, the duality isomorphism at the level of $\cD$-modules and the one at the level of Stokes-filtered local systems are constructed in parallel ways.

\num{}
The correspondence between the duality isomorphism for proper maps, in both settings, is proved in \cite{MSaito89b}. A similar statement for pull-back by a smooth morphism can be proved similarly.

\num{}
This reduces the problem to a comparison between both isomorphisms on $\PP^1_t\times\nobreak\PP^1_{\hb'}$. For our purpose, one can show that it is enough to compare both isomorphisms over $\PP^1_t\times\CC^*_{\hb'}$. On the one hand, we have $D(p^+M\otimes E^{-t\hb'})\simeq p^+(DM)\otimes E^{t\hb'}$ (the change of sign explaining the need of $\iota$). On the other hand, let $Z=\PP^1_t\times\CC^*_{\hb'}$ and $\varpi: \wt Z\to Z$ be the oriented real blowing-up along $\{\infty\}\times\CC^*_{\hb'}$. If $\cF$ is the perverse sheaf associated to $M$ on $\Afu_t$ via the de~Rham functor, the perverse sheaf on $Z$ associated to $p^+M\otimes E^{-t\hb'}$ via the de~Rham functor can be written as $\bR\varpi_*\beta_!\bR\alpha_*(p\circ j)^{-1}\cF[1]$, where, setting
\[
L'_{\leq0}=\{(t,\hb')\mid t=0\text{ or }\arg t+\arg\hb'\in[\pi/2,3\pi/2]\bmod2\pi\}\subset\wt Z,
\]
$\alpha,\beta$ are the inclusions
\[
\Afu_t\times\CC^*_{\hb'}\Hto{\alpha}L'_{\leq0}\Hto{\beta}\wt Z.
\]
and $j=\varpi\circ\beta\circ\alpha$ is the inclusion $\Afu_t\times\CC^*_{\hb'}\hto Z$. The question is then reduced to a local duality theorem, comparing the duality isomorphism for $\cD$-modules and the Poincaré-Verdier duality for perverse sheaves.\qed

\backmatter
\newcommand{\SortNoop}[1]{}\def\cprime{$'$}
\providecommand{\bysame}{\leavevmode ---\ }
\providecommand{\og}{``}
\providecommand{\fg}{''}
\providecommand{\smfandname}{\&}
\providecommand{\smfedsname}{\'eds.}
\providecommand{\smfedname}{\'ed.}
\providecommand{\smfmastersthesisname}{M\'emoire}
\providecommand{\smfphdthesisname}{Th\`ese}


\begin{thebibliography}{10}

\bibitem{B-J-L81}
{\scshape W.~Balser, {\relax W.B}.~Jurkat {\normalfont \smfandname} {\relax
  D.A}.~Lutz} -- {\og On the reduction of connection problems for differential
  equations with an irregular singular point to ones with only regular
  singularities. {I}\fg}, \emph{SIAM J.~Math. Anal.} \textbf{12} (1981), no.~5,
  p.~691--721.

\bibitem{Borel87b}
{\scshape A.~Borel} -- {\og {Chap.\ VI-IX}\fg}, in \emph{Algebraic
  {$\mathcal{D}$}-modules}, Perspectives in Math., vol.~2, Academic Press,
  Boston, 1987, p.~207--352.

\bibitem{C-K-S86}
{\scshape E.~Cattani, A.~Kaplan {\normalfont \smfandname} W.~Schmid} -- {\og
  Degeneration of {Hodge} structures\fg}, \emph{Ann. of Math.} \textbf{123}
  (1986), p.~457--535.

\bibitem{C-F-I-V92}
{\scshape S.~Cecotti, P.~Fendley, K.~Intriligator {\normalfont \smfandname}
  C.~Vafa} -- {\og {A new supersymmetric index}\fg}, \emph{Nuclear Phys. B}
  \textbf{386} (1992), p.~405--452.

\bibitem{C-V91}
{\scshape S.~Cecotti {\normalfont \smfandname} C.~Vafa} -- {\og
  {Topological-antitopological fusion}\fg}, \emph{Nuclear Phys. B} \textbf{367}
  (1991), p.~359--461.

\bibitem{C-V93}
\bysame , {\og {On classification of $N=2$ supersymmetric theories}\fg},
  \emph{Comm. Math. Phys.} \textbf{158} (1993), p.~569--644.

\bibitem{Deligne78}
{\scshape P.~Deligne} -- {\og {Lettre \`a B.~Malgrange du 19/4/1978}\fg}, in
  \emph{{Singularit\'es irr\'eguli\`eres, Correspondance et documents}},
  Documents math\'ematiques, vol.~5, Soci{\'e}t{\'e} Math{\'e}matique de
  France, Paris, 2007, p.~25--26.

\bibitem{Deligne8406}
\bysame , {\og {Th\'eorie de Hodge irr\'eguli\`ere (mars 1984 \& ao{\^u}t
  2006)}\fg}, in \emph{{Singularit\'es irr\'eguli\`eres, Correspondance et
  documents}}, Documents math\'ematiques, vol.~5, Soci{\'e}t{\'e}
  Math{\'e}matique de France, Paris, 2007, p.~109--114 \& 115--128.

\bibitem{Hertling00}
{\scshape C.~Hertling} -- \emph{{Frobenius manifolds and moduli spaces for
  singularities}}, Cambridge Tracts in Mathematics, vol. 151, Cambridge
  University Press, 2002.

\bibitem{Hertling01}
\bysame , {\og {$tt^*$ geometry, Frobenius manifolds, their connections, and
  the construction for singularities}\fg}, \emph{J.~reine angew. Math.}
  \textbf{555} (2003), p.~77--161.

\bibitem{Hertling06}
\bysame , {\og {$tt\sp *$} geometry and mixed {H}odge structures\fg}, in
  \emph{Singularity theory and its applications}, Adv. Stud. Pure Math.,
  vol.~43, Math. Soc. Japan, Tokyo, 2006, p.~73--84.

\bibitem{H-S09}
{\scshape C.~Hertling {\normalfont \smfandname} C.~Sabbah} -- {\og {Examples of
  non-commutative Hodge structures}\fg}, \emph{Journal de l'Institut
  mathématique de Jussieu} \textbf{10} (2011), no.~3, p.~635--674.

\bibitem{H-S06}
{\scshape C.~Hertling {\normalfont \smfandname} {\relax Ch}.~Sevenheck} -- {\og
  {Nilpotent orbits of a generalization of Hodge structures}\fg},
  \emph{J.~reine angew. Math.} \textbf{609} (2007), p.~23--80, arXiv:
  \url{math.AG/0603564}.

\bibitem{H-S07}
\bysame , {\og Curvature of classifying spaces for {B}rieskorn lattices\fg},
  \emph{J.~Geom. Phys.} \textbf{58} (2008), no.~11, p.~1591--1606, arXiv:
  \url{0712.3691}.

\bibitem{H-S08b}
\bysame , {\og {Twistor structures, tt*-geometry and singularity theory}\fg},
  in \emph{{From Hodge theory to integrability and TQFT: tt*-geometry}}
  (R.~Donagi {\normalfont \smfandname} K.~Wendland, \smfedsname), Proc.
  Symposia in Pure Math., vol.~78, American Mathematical Society, Providence,
  RI, 2008, p.~49--73, arXiv: \url{0807.2199}.

\bibitem{H-S08}
\bysame , {\og {Limits of families of Brieskorn lattices and compactified
  classifying spaces}\fg}, \emph{Adv. in Math.} \textbf{223} (2010),
  p.~1155--1224, arXiv: \url{0805.4777}.

\bibitem{Iritani09}
{\scshape H.~Iritani} -- {\og An integral structure in quantum cohomology and
  mirror symmetry for toric orbifolds\fg}, \emph{Adv. in Math.} \textbf{222}
  (2009), no.~3, p.~1016--1079, arXiv: \url{0903.1463v3}.

\bibitem{Iritani09b}
\bysame , {\og {tt*-geometry in quantum cohomology}\fg}, arXiv:
  \url{0906.1307}, 2009.

\bibitem{Kaledin08}
{\scshape D.~Kaledin} -- {\og Non-commutative {H}odge-to-de {R}ham degeneration
  via the method of {D}eligne-{I}llusie\fg}, \emph{Pure Appl. Math.~Q.}
  \textbf{4} (2008), no.~3, Special issue in honor of F.~Bogomolov, Part 2,
  p.~785--875.

\bibitem{Kaledin10}
\bysame , {\og Motivic structures in non-commutative geometry\fg}, arXiv:
  \url{1003.3210}; to appear in Proc. ICM, 2010.

\bibitem{Kashiwara85}
{\scshape M.~Kashiwara} -- {\og {The asymptotic behaviour of a variation of
  polarized Hodge structure}\fg}, \emph{Publ. RIMS, Kyoto Univ.} \textbf{21}
  (1985), p.~853--875.

\bibitem{Kashiwara03}
\bysame , \emph{{$D$}-modules and microlocal calculus}, Translations of
  Mathematical Monographs, vol. 217, American Mathematical Society, Providence,
  RI, 2003.

\bibitem{Katz90}
{\scshape N.~Katz} -- \emph{Exponential sums and differential equations}, Ann.
  of Math. studies, vol. 124, Princeton University Press, Princeton, NJ, 1990.

\bibitem{Katz96}
\bysame , \emph{{Rigid local systems}}, Ann. of Math. studies, vol. 139,
  Princeton University Press, Princeton, NJ, 1996.

\bibitem{K-L85}
{\scshape N.~Katz {\normalfont \smfandname} G.~Laumon} -- {\og {Transformation
  de Fourier et majoration de sommes exponentielles}\fg}, \emph{Publ. Math.
  Inst. Hautes {\'E}tudes Sci.} \textbf{62} (1985), p.~145--202.

\bibitem{K-K-P08}
{\scshape L.~Katzarkov, M.~Kontsevich {\normalfont \smfandname} T.~Pantev} --
  {\og {Hodge theoretic aspects of mirror symmetry}\fg}, in \emph{{From Hodge
  theory to integrability and TQFT: tt*-geometry}} (R.~Donagi {\normalfont
  \smfandname} K.~Wendland, \smfedsname), Proc. Symposia in Pure Math.,
  vol.~78, American Mathematical Society, Providence, RI, 2008, arXiv:
  \url{0806.0107}, p.~87--174.

\bibitem{K-S10}
{\scshape M.~Kontsevich {\normalfont \smfandname} Y.~Soibelman} -- {\og
  {Cohomological Hall algebra, exponential Hodge structures and motivic
  Donaldson-Thomas invariants}\fg}, arXiv: \url{1006.2706}, 2010.

\bibitem{Malgrange91}
{\scshape B.~Malgrange} -- \emph{{\'E}quations diff\'erentielles {\`a}
  coefficients polynomiaux}, Progress in Math., vol.~96, Birkh{\"a}user, Basel,
  Boston, 1991.

\bibitem{Malgrange95}
\bysame , {\og {Connexions m\'eromorphes, II: le r\'eseau canonique}\fg},
  \emph{Invent. Math.} \textbf{124} (1996), p.~367--387.

\bibitem{Mochizuki07}
{\scshape T.~Mochizuki} -- \emph{{Asymptotic behaviour of tame harmonic bundles
  and an application to pure twistor $D$-modules}}, vol. 185, Mem. Amer. Math.
  Soc., no. 869-870, American Mathematical Society, Providence, RI, 2007,
  arXiv: \url{math.DG/0312230} \& \url{math.DG/0402122}.

\bibitem{Mochizuki10}
\bysame , {\og {Holonomic $\mathcal D$-modules with Betti structure}\fg},
  arXiv: \url{1001.2336}, 2010.

\bibitem{Mochizuki08b}
\bysame , {\og {Asymptotic behaviour of variation of pure polarized TERP
  structures}\fg}, \emph{Publ. RIMS, Kyoto Univ.} \textbf{47} (2011), no.~2,
  p.~419--534, arXiv: \url{0811.1384}.

\bibitem{Mochizuki08}
\bysame , \emph{{Wild harmonic bundles and wild pure twistor $D$-modules}},
  Ast{\'e}risque, Soci{\'e}t{\'e} Math{\'e}matique de France, Paris, 2011, to
  appear; arXiv: \url{0803.1344}.

\bibitem{P-S08}
{\scshape C.~Peters {\normalfont \smfandname} {\relax J.H.M}.~Steenbrink} --
  \emph{{Mixed Hodge Structures}}, Ergebnisse der Mathematik und ihrer
  Grenzgebiete. 3. Folge, vol.~52, Springer-Verlag, Berlin, 2008.

\bibitem{Bibi96b}
{\scshape C.~Sabbah} -- {\og Hypergeometric period for a tame polynomial\fg},
  \emph{C.~R. Acad. Sci. Paris S{\'e}r. I Math.} \textbf{328} (1999),
  p.~603--608, \& \url{math.AG/9805077}.

\bibitem{Bibi00}
\bysame , \emph{{D\'eformations isomonodromiques et vari\'et\'es de
  Frobenius}}, Savoirs Actuels, CNRS~{\'E}ditions \& EDP~Sciences, Paris, 2002,
  English Transl.: Universitext, Springer \& EDP~Sciences, 2007.

\bibitem{Bibi01c}
\bysame , \emph{{Polarizable twistor $\mathcal{D}$-modules}}, Ast{\'e}risque,
  vol. 300, Soci{\'e}t{\'e} Math{\'e}matique de France, Paris, 2005.

\bibitem{Bibi96bb}
\bysame , {\og Hypergeometric periods for a tame polynomial\fg},
  \emph{Portugal. Math.} \textbf{63} (2006), no.~2, p.~173--226, arXiv:
  \url{math.AG/9805077}.

\bibitem{Bibi05}
\bysame , {\og {Fourier-Laplace transform of a variation of polarized complex
  Hodge structure}\fg}, \emph{J.~reine angew. Math.} \textbf{621} (2008),
  p.~123--158, arXiv: \url{math.AG/0508551}.

\bibitem{Bibi06b}
\bysame , {\og {Wild twistor $\mathcal D$-modules}\fg}, in \emph{{Algebraic
  Analysis and Around: In Honor of Professor Masaki Kashiwara's 60th Birthday
  (Kyoto, June 2007)}}, Advanced Studies in Pure Math., vol.~54, Math. Soc.
  Japan, Tokyo, 2009, p.~293--353, arXiv: \url{0803.0287}.

\bibitem{Bibi08}
\bysame , {\og {Fourier-Laplace transform of a variation of polarized complex
  Hodge structure, II}\fg}, in \emph{{New developments in Algebraic Geometry,
  Integrable Systems and Mirror symmetry (Kyoto, January 2008)}}, Advanced
  Studies in Pure Math., vol.~59, Math. Soc. Japan, Tokyo, 2010, p.~289--347,
  arXiv: \url{0804.4328}.

\bibitem{Bibi10}
\bysame , {\og {Introduction to Stokes structures}\fg}, Lecture Notes (Lisboa,
  January 2009) 200 pages, arXiv: \url{0912.2762}, 2010.

\bibitem{MSaito86}
{\scshape M.~Saito} -- {\og Modules de {Hodge} polarisables\fg}, \emph{Publ.
  RIMS, Kyoto Univ.} \textbf{24} (1988), p.~849--995.

\bibitem{MSaito89b}
\bysame , {\og Induced {$\mathcal{D}$}-modules and differential complexes\fg},
  \emph{Bull. Soc. math. France} \textbf{117} (1989), p.~361--387.

\bibitem{MSaito89}
\bysame , {\og {On the structure of Brieskorn lattices}\fg}, \emph{Ann. Inst.
  Fourier (Grenoble)} \textbf{39} (1989), p.~27--72.

\bibitem{MSaito87}
\bysame , {\og {Mixed {Hodge} Modules}\fg}, \emph{Publ. RIMS, Kyoto Univ.}
  \textbf{26} (1990), p.~221--333.

\bibitem{Shklyarov11}
{\scshape D.~Shklyarov} -- {\og Non-commutative hodge structures: towards
  matching categorical and geometric examples\fg}, arXiv: \url{1107.3156},
  2011.

\bibitem{Simpson90}
{\scshape C.~Simpson} -- {\og Harmonic bundles on noncompact curves\fg},
  \emph{J.~Amer. Math. Soc.} \textbf{3} (1990), p.~713--770.

\bibitem{Simpson91}
\bysame , {\og Nonabelian {H}odge theory\fg}, in \emph{Proceedings of the
  {I}nternational {C}ongress of {M}athematicians, ({K}yoto, 1990)}, Math. Soc.
  Japan, Tokyo, 1991, p.~747--756.

\bibitem{Simpson92}
\bysame , {\og Higgs bundles and local systems\fg}, \emph{Publ. Math. Inst.
  Hautes {\'E}tudes Sci.} \textbf{75} (1992), p.~5--95.

\bibitem{Simpson97}
\bysame , {\og Mixed twistor structures\fg}, Pr\'epublication Universit\'e de
  Toulouse \& arXiv: \url{math.AG/9705006}, 1997.

\bibitem{Simpson97b}
\bysame , {\og {The Hodge filtration on nonabelian cohomology}\fg}, in
  \emph{Algebraic geometry (Santa Cruz, 1995)}, Proc. of AMS summer
  conferences, American Mathematical Society, 1997, p.~217--281.

\bibitem{Simpson02}
\bysame , {\og Algebraic aspects of higher nonabelian {H}odge theory\fg}, in
  \emph{Motives, polylogarithms and {H}odge theory, {P}art {II} ({I}rvine,
  {CA}, 1998)}, Int. Press Lect. Ser., vol.~3, Int. Press, Somerville, MA,
  2002, p.~417--604.

\bibitem{Steenbrink87}
{\scshape {\relax J.H.M}.~Steenbrink} -- {\og The spectrum of hypersurface
  singularities\fg}, in \emph{Th\'eorie de {Hodge} (Luminy, 1987)},
  Ast{\'e}risque, vol. 179-180, Soci{\'e}t{\'e} Math{\'e}matique de France,
  1989, p.~163--184.

\bibitem{Varchenko82}
{\scshape {\relax A.N}.~Varchenko} -- {\og {Asymptotic Hodge structure on the
  cohomology of the Milnor fiber}\fg}, \emph{Math. USSR Izv.} \textbf{18}
  (1982), p.~469--512.

\end{thebibliography}
\end{document}